\documentclass[12pt,reqno]{amsart}  

\usepackage{amsmath,amsthm,amssymb,amscd,mathdots,mathtools,enumerate,shuffle,stmaryrd,faktor}
\usepackage{mathabx} 
\usepackage{hyperref}
\usepackage{tikz,tikz-cd}
\allowdisplaybreaks
\usetikzlibrary{decorations.pathmorphing,shapes,calc}

\newtheorem{thm}{Theorem}[section]
\newtheorem*{thm*}{Theorem}
\newtheorem{prop}[thm]{Proposition}

\newtheorem{cor}[thm]{Corollary}
\newtheorem{con}[thm]{Conjecture}
\numberwithin{equation}{section}

\theoremstyle{definition}
\newtheorem{defi}[thm]{Definition}
\newtheorem*{defi*}{Definition}
\newtheorem{ex}[thm]{Example}
\newtheorem{rem}[thm]{Remark}

\mathtoolsset{showonlyrefs} 

\newcommand{\QQ}{\mathbb{Q}}
\newcommand{\ZZ}{\mathbb{Z}}

\newcommand{\diff}{\mathop{}\!\mathrm{d}}
\newcommand{\nca}[2]{#1\langle #2\rangle}
\newcommand{\Qnca}[1]{\QQ\langle #1 \rangle}
\newcommand{\ncac}[2]{#1\langle \langle #2\rangle\rangle}
\newcommand{\spanQ}{\operatorname{span}_{\QQ}}
\newcommand{\Sl}{\operatorname{SL}_2(\ZZ)}
\newcommand{\quasimod}{\widetilde{\mathcal{M}}^\QQ(\Sl)}

\newcommand{\1}{\mathbf{1}}
\newcommand{\quasish}{\ast_\diamond}
\newcommand{\dec}{\Delta_{\operatorname{dec}}}
\newcommand{\A}{\mathcal{A}}
\newcommand{\B}{\mathcal{B}}
\newcommand{\QB}{\nca{\QQ}{\mathcal{B}}}
\newcommand{\Ybi}{\mathcal{Y}^{\operatorname{bi}}}
\newcommand{\QYbi}{\nca{\QQ}{\Ybi}}
\newcommand{\Y}{\mathcal{Y}}
\newcommand{\QY}{\nca{\QQ}{\mathcal{Y}}}
\newcommand{\X}{\mathcal{X}}
\newcommand{\QX}{\nca{\QQ}{\mathcal{X}}}
\newcommand{\h}{\mathfrak{h}}

\newcommand{\Gen}{\rho_{\A}(\mathcal{W})}
\newcommand{\Genphi}{\varphi(\mathcal{W}(\A))}
\newcommand{\GenY}{\rho_{\Y}(\mathcal{W})}

\newcommand{\GenYbi}{\rho_{\Ybi}(\mathcal{W})}
\newcommand{\GenBo}{\rho_{\B}^0(\mathcal{W}^0)}
\newcommand{\GenB}{\rho_{\B}(\mathcal{W})}
\newcommand{\GenR}{(\varphi\otimes \rho_{\A})(\mathcal{W})}
\newcommand{\GenRqs}{(\varphi_{\quasish}\otimes \rho_{\A})(\mathcal{W})}

\newcommand{\wt}{\operatorname{wt}}
\newcommand{\dep}{\operatorname{dep}} 

\newcommand{\Z}{\mathcal{Z}}

\newcommand{\Zq}{\mathcal{Z}_q}
\newcommand{\zq}{\zeta_q}

\newcommand{\bsh}{\shuffle_b}
\newcommand{\bst}{\ast_b}
\newcommand{\qdia}{\diamond_b}
\newcommand{\regT}{\operatorname{reg}_T}
\newcommand{\reg}{\operatorname{reg}}
\newcommand{\cmes}[2]{G\binom{#1}{#2}}

\newcommand{\bi}[2]{\binom{#1}{#2}}

\newcommand{\Gzq}[3]{\mathfrak{B}_{#1}\binom{#2}{#3}}
\newcommand{\Gcmes}[3]{\mathfrak{G}_{#1}\binom{#2}{#3}}

\newcommand{\swap}{\operatorname{swap}}

\newcommand{\todo}[1]{{\color{red} To do: #1}}

\title{Balanced multiple $\operatorname{q}$-zeta values}

\author{Annika Burmester}
\address{Faculty of Mathematics, Bielefeld University, Germany.}
\email{aburmester@math.uni-bielefeld.de}



\thanks{The author was partially funded by the Deutsche Forschungsgemeinschaft (DFG, German Research Foundation) -- SFB-TRR 358/1 2023 — 491392403.}

\begin{document} 
	
\maketitle
	
\begin{abstract} \noindent
We introduce the balanced multiple q-zeta values. They give a new model for multiple q-zeta values, whose product formula combines the shuffle and stuffle product for multiple zeta values in a natural way. Moreover, the balanced multiple q-zeta values are invariant under a very explicit involution. Thus, all relations among the balanced multiple q-zeta values are conjecturally of a very simple shape. Examples of the balanced multiple q-zeta values are the classical Eisenstein series, and they also contain the combinatorial multiple Eisenstein series introduced in \cite{bb}. The construction of the balanced multiple q-zeta values is done on the level of generating series. We introduce a general setup relating Hoffman's quasi-shuffle products to explicit symmetries among generating series of words, which gives a clarifying approach to Ecalle's theory of bimoulds. This allows us to obtain an isomorphism between the underlying Hopf algebras of words related to the combinatorial bi-multiple Eisenstein series and the balanced multiple q-zeta values.
\end{abstract} 

	
\section{Introduction} 

\noindent
In general, a \emph{q-analog} of some expression is a generalization involving the variable $q$, which returns the original expression by taking the limit $q\to 1$. For example, a q-analog of some integer $n\geq 1$ is 
\[\{n\}_q=\frac{1-q^n}{1-q}=1+q+\dots+q^{n-1},\]
since $\lim_{q\to 1} \{n\}_q=n$. In this work we are interested in a particular kind of q-analogs of multiple zeta values.

\noindent\vspace{-0,5cm} \\
\paragraph{\textbf{Multiple zeta values.}} \emph{Multiple zeta values} are real numbers defined for integers $k_1\geq2,\ k_2,\ldots,k_d\geq1$ by
\[\zeta(k_1,\ldots,k_d)=\sum_{n_1>\dots > n_d>0} \frac{1}{n_1^{k_1}\dots n_d^{k_d}}.\]
We refer to the number $k_1+\dots+k_d$ as the \emph{weight} and to the number $d$ as the \emph{depth}. Multiple zeta values were first studied in depth $2$ by C. Goldbach and L. Euler more than two centuries ago. During the last three decades, they became an active field of research. They occur in various areas of mathematics, such as number theory, algebraic geometry, knot theory, quantum field theory, and also in mathematical physics. Multiple zeta values appeared in full generality for the first time in \cite[p. 429]{ec-old}, and the modern systematic study of them was initiated in \cite{h2}, \cite{za}. 
A survey on achievements in the theory of multiple zeta values can be found in \cite{bgf} and \cite{zh}, and all articles related to multiple zeta values are listed in \cite{hw}. 
\vspace{0,2cm} \\
The product of multiple zeta values can be expressed in two different ways, one is called the \emph{stuffle product}, and the other one is called the \emph{shuffle product}. Both products possess a description in terms of quasi-shuffle algebras (introduced in Section \ref{quasi-shuffle algebras}). First, consider the alphabet $\Y=\{y_1,y_2,\ldots\}$ and let $\QY$ be the non-commutative $\QQ$-algebra generated by $\Y$. Then the \emph{stuffle product} $\ast$ on $\QY$ is defined by $\1\ast w=w\ast \1=w$ and
\begin{align} \label{stuffle product}
y_iu\ast y_jv=y_i(u\ast y_jv)+y_j(y_iu\ast v)+y_{i+j}(u\ast v)
\end{align} 
for all $u,v,w\in \QY$. The combinatorics of infinite nested sums imply that there is a surjective algebra morphism
\begin{align} \label{stuffle MZV}
(\QY,\ast)&\to \Z, \\
y_{k_1}\dots y_{k_d}&\mapsto \zeta^{\ast}(k_1,\ldots,k_d). \nonumber
\end{align}
The elements $\zeta^{\ast}(k_1,\ldots,k_d)$ are the \emph{stuffle regularized multiple zeta values}, they are uniquely determined by $\zeta^{\ast}(k_1,\ldots,k_d)=\zeta(k_1,\ldots,k_d)$ for all $k_1\geq2,\ k_2,\ldots,k_d\geq1$ and $\zeta^{\ast}(1)=0$. \\
Second, consider the finite alphabet $\X=\{x_0,x_1\}$ and let $\QX$ be the non-commutative algebra over $\QQ$ generated by $\X$. The \emph{shuffle product} on $\QX$ is recursively defined by $\1\shuffle w=w\shuffle \1=w$ and 
\begin{align} \label{shuffle product}
x_iu\shuffle x_jv=x_i(u\shuffle x_jv)+x_j(x_iu\shuffle v)
\end{align}
for all $u,v,w\in \QX$. Define $\h^1$ to be the subspace of $\QX$ generated by all words not ending in $x_0$, i.e., we have $\h^1=\QQ\1+\QX x_1$. Using the iterated integral representation of multiple zeta values one obtains a surjective algebra morphism
\begin{align} \label{shuffle MZV}
(\h^1,\shuffle)&\to \Z, \\
x_0^{k_1-1}x_1\dots x_0^{k_d-1}x_1&\mapsto \zeta^{\shuffle}(k_1,\ldots,k_d). \nonumber
\end{align}
The elements $\zeta^{\shuffle}(k_1,\ldots,k_d)$ are the \emph{shuffle regularized multiple zeta values}, they are uniquely determined by $\zeta^{\shuffle}(k_1,\ldots,k_d)=\zeta(k_1,\ldots,k_d)$ for all $k_1\geq2,\ k_2,\ldots,k_d\geq 1$ and $\zeta^{\shuffle}(1)=0$. \\
In \cite[Theorem 1]{ikz} an explicit map is given relating the stuffle and the shuffle regularized multiple zeta values, which allows comparing the two product formulas given in \eqref{stuffle MZV} and \eqref{shuffle MZV}. This yields the \emph{extended double shuffle relations} among multiple zeta values, which are conjectured to give all algebraic relations in $\Z$ (\cite{ikz}). \\
Quasi-shuffle products can be lifted to symmetries among generating series (Section \ref{quasi-shuffle and gen series}). Consider the generating series of the shuffle resp. stuffle regularized multiple zeta values
\begin{align*}
\mathfrak{z}_d^{\bullet}(X_1,\ldots,X_d)&=\sum_{k_1,\ldots,k_d\geq1} \zeta^{\bullet}(k_1,\ldots,k_d)X_1^{k_1-1}\dots X_d^{k_d-1}\in \Z\llbracket X_1,\ldots,X_d\rrbracket, \qquad d\geq1,
\end{align*}
with $\bullet\in \{\ast,\shuffle\}$. The stuffle and shuffle product in depth $2$ translate to
\begin{align} \label{product MZV gen series}
\mathfrak{z}_1^{\ast}(X_1)\mathfrak{z}_1^{\ast}(X_2)&=\mathfrak{z}_2^{\ast}(X_1,X_2)+\mathfrak{z}_2^{\ast}(X_2,X_1)+\frac{\mathfrak{z}_1^{\ast}(X_1)-\mathfrak{z}_1^{\ast}(X_2)}{X_1-X_2}, \\
\mathfrak{z}_1^{\shuffle}(X_1)\mathfrak{z}_1^{\shuffle}(X_2)&=\mathfrak{z}_2^{\shuffle}(X_1+X_2,X_2)+\mathfrak{z}_2^{\shuffle}(X_1+X_2,X_1).
\end{align}
In general depths, there are recursive formulas for the stuffle and shuffle product in terms of generating series of words (\cite[Proposition 8]{ih}). Those product formulas explain Ecalle's notion of \emph{symmetril} and \emph{symmetral} moulds (\cite[equation (6), (8)]{ec-mzv}).

\noindent\vspace{-0,5cm} \\
\paragraph{\textbf{Multiple q-zeta values.}} For a better understanding of the structure of the multiple zeta values, we study particular q-analogs of them. Including the additional variable $q$ reveals new structures. For example, one obtains an involution relating the two product formulas of multiple zeta values, or a derivation with respect to $q$ known from the theory of quasi-modular forms.
\vspace{0,2cm}\\ 
We use the model-free approach to q-analogs of multiple zeta values given in \cite{bk}. To integers $s_1\geq 1,\ s_2,\ldots,s_l\geq 0$ and polynomials $R_1\in t\QQ[t],\ R_2,\ldots,R_l\in$ $\QQ[t]$, associate the \emph{generic multiple q-zeta value}
\[\zeta_q(s_1,...,s_l;R_1,...,R_l)=\sum_{n_1>\dots>n_l>0} \frac{R_1(q^{n_1})}{(1-q^{n_1})^{s_1}}\cdots \frac{R_l(q^{n_l})}{(1-q^{n_l})^{s_l}} \in \QQ\llbracket q\rrbracket.\]
For $s_1\geq2$, $s_2,\ldots,s_l\geq1$, a generic multiple q-zeta values $\zq(s_1,\ldots,s_l;R_1,\ldots,R_l)$ is indeed a (modified) q-analog of multiple zeta values (Proposition \ref{limit generic qMZV}). The product of any two generic multiple q-zeta values is a $\QQ$-linear combination of generic multiple q-zeta values.
\begin{defi*} The \emph{algebra of multiple q-zeta values} is the subalgebra of $\QQ\llbracket q\rrbracket$ given by
\[\Zq=\spanQ\!\big\{\zq(s_1,...,s_l;R_1,...,R_l)\ \big|\ l\geq 0,s_1\geq 1,\ s_2,...,s_l\geq 0,\deg(R_j)\leq s_j\big\},\]
where we set $\zeta_q(\emptyset;\emptyset)=1$.
\end{defi*} \noindent
The additional requirement on the degree of the polynomials is natural. For example, $\Zq$ is related to polynomial functions on partitions (\cite{ba}, \cite{bi}), which implies the existence of spanning sets of $\Zq$ invariant under some involution. The space $\Zq$ also occurs in enumerative geometry. More precisely, A. Okounkov conjectured that certain generating series of Chern numbers on Hilbert schemes of points are always contained in the space $\Zq$ (\cite{ok}). 
\vspace{0,2cm}\\
A \emph{model} of multiple q-zeta values is a particular assumption on the polynomials $R_i$ (usually these particular polynomials form a basis of $\QQ[t]$) or just a spanning set of $\Zq$. About twenty years ago the first models for multiple q-zeta values were introduced independently by D. Bradley, K. Schlesinger, J. Zhao, and W. Zudilin (\cite{bra}, \cite{schl}, \cite{zh2}, \cite{zu2}). Since then several more models appeared in the literature, all of them are contained in the above defined space $\Zq$. An overview on the various models of multiple q-zeta values and their relations is given in \cite{bk} and \cite{bri}. \\
In the first articles about multiple q-zeta values the focus lies on a q-analog of the shuffle product given in \eqref{stuffle MZV} or the stuffle product given in \eqref{shuffle MZV} for multiple zeta values but not on the combination of them. Later, q-analogs of the double shuffle relation became of more interest and were developed by using q-analogs of multiple polylogarithms (\cite{ta}) or Rota-Baxter operators (\cite{cem}). In rather recent articles on multiple q-zeta values (\cite{ba}, \cite{ems}) the focus changed to obtaining a product formula and some invariance under an involution. Combining the product formula and the involution one easily derives a q-analog of the extended double shuffle relations. In joint work with H. Bachmann (\cite{bb}), we constructed the combinatorial bi-multiple Eisenstein series, which satisfy a weight-graded product formula and are invariant under some weight-homogeneous involution. 

\noindent \vspace{-0,5cm} \\
\paragraph{\textbf{Balanced multiple q-zeta values.}} 
A main result of this article gives a new model of multiple q-zeta values, which also satisfies weight-homogeneous relations and is closely related to the Schlesinger-Zudilin multiple q-zeta values studied in \cite{ems}. We call the elements in this model the \emph{balanced multiple q-zeta values} $\zq(s_1,\ldots,s_l),\ s_1\geq1,\ s_2,\ldots,s_l\geq0$ (Definition \ref{def balanced qMZV}). The classical Eisenstein series and their derivatives are examples of balanced multiple q-zeta values. Moreover, the balanced multiple q-zeta values can be used to conjecturally describe all relations among multiple Eisenstein series. This new model also gives an explicit description of a conjectural weight-grading on the algebra $\Zq$, which extends the weight-grading of the algebra $\quasimod$ of quasi-modular forms with rational coefficients (introduced in \cite{kz}). The balanced multiple q-zeta values are q-analogs of multiple zeta values, taking the limit $q\to 1$ always yields an element in the algebra of multiple zeta values (Proposition \ref{non-regularized limit balanced qMZV}, Remark \ref{regularized limit balanced qMZV}). An advantage of the balanced multiple q-zeta values is that they satisfy a product formula, which can be seen as a balanced combination of the shuffle and stuffle product for the multiple zeta values (Theorem \ref{balanced qMZV quasi-shuffle morphism}). Therefore, they provide a setup to study both products of the multiple zeta values at the same time. Moreover, the balanced multiple q-zeta values satisfy linear relations coming from a particular simple involution. They also possess a description in terms of a finite alphabet similar to the case of multiple zeta values (Theorem \ref{grSZ q-shuffle}).
\vspace{0,2cm} \\
Consider the alphabet 
\[\B=\{b_0,b_1,b_2,\ldots\}\] 
and denote by $\QB$ the free non-commutative $\QQ$-algebra generated by $\B$. The \emph{balanced quasi-shuffle product} $\bst$ is the quasi-shuffle product on $\QB$ recursively defined by $\1\bst w=w\bst \1=w$ and
\begin{align} \label{def q-stuffle intro}
b_iu\bst b_jv=b_i(u\bst b_jv)+b_j(b_iu\bst v)+\begin{cases} b_{i+j}(u\bst v) & \text{ if } i,j\geq 1, \\ 0 & \text{ else}
\end{cases}
\end{align}
for all $b_i,b_j\in \B,\ u,v,w\in \QB$. Restricted to the letters $b_i$ for $i\geq 1$ the product $\bst$ is the stuffle product given in \eqref{stuffle product} and modulo words containing the letters $b_i$ for $i\geq2$ we recover the shuffle product given in \eqref{shuffle product}, therefore we call $\bst$ the balanced quasi-shuffle product. The product $\bst$ is exactly the associated weight-graded product to the quasi-stuffle product of the Schlesinger-Zudilin multiple q-zeta values (\cite[Theorem 3.3]{si}). \\
Denote by $\QB^0$ the subalgebra of $\QB$ generated by all words which do not start in $b_0$, i.e., we have $\QB^0=\QB\backslash b_0\QB$. Let $\tau:\QB^0\to \QB^0$ be the involution given by $\tau(\1)=\1$ and
\begin{align*}
\tau(b_{k_1}b_0^{m_1}\dots b_{k_d}b_0^{m_d})=b_{m_d+1}b_0^{k_d-1}\dots b_{m_1+1}b_0^{k_1-1}.
\end{align*} 
A main result of this work is the following.
\begin{thm*}[Theorem \ref{balanced qMZV quasi-shuffle morphism}] \label{intro graded SZ qMZV} There is a $\tau$-invariant, surjective algebra morphism
\begin{align*}
(\QB^0, \bst)&\to\Zq, \\
b_{s_1}\dots b_{s_l}&\mapsto \zq(s_1,\ldots,s_l), 
\end{align*}
where $\zq(s_1,\ldots,s_l)$ are the balanced multiple q-zeta values introduced in Definition \ref{def balanced qMZV}.
\end{thm*} \noindent
To the balanced multiple q-zeta value $\zq(s_1,\ldots,s_l)$ associate the \emph{weight}
\[s_1+\dots+s_l+\#\{i\mid s_i=0\}.\] 
The balanced quasi-shuffle product as well as the $\tau$-invariance of the balanced multiple q-zeta values are homogeneous for the weight. Since we expect that all algebraic relations in $\Zq$ can be deduced from these to sets of relations among balanced multiple q-zeta values (Conjecture \ref{all relations in Zq balanced qMZV}), we expect the algebra $\Zq$ to be graded with respect to this weight.
\vspace{0,2cm} \\
Essential for the construction of the balanced multiple q-zeta values are the \emph{combinatorial bi-multiple Eisenstein series} introduced in \cite{bb} (Definition \ref{def cmes}), which also possess a description in terms of quasi-shuffle algebras. Consider the alphabet $\Ybi=\{y_{k,m}\mid k\geq1, m\geq0\}$ and let $\QYbi$ be the non-commutative $\QQ$-algebra generated by $\Ybi$. The \emph{stuffle product} $\ast$ on $\QYbi$ is recursively defined by $\1\ast w=w\ast \1=w$ and
\[y_{k_1,m_2}u\ast y_{k_2,m_2}v=y_{k_1,m_1}(u\ast y_{k_2,m_2}v)+y_{k_2,m_2}(y_{k_1,m_1}u\ast v)+y_{k_1+k_2,m_1+m_2}(u\ast v)\]
for all $u,v,w\in \QYbi$. This is the canonical bi-version of the stuffle product given in \eqref{stuffle product}. Moreover, there is an involution $\swap:\QYbi\to \QYbi$ (Definition \ref{def swap}) closely related to conjugation of partitions. \\
Following the general explanations in Section \ref{quasi-shuffle and gen series}, we will express the stuffle product $\ast$ on $\QYbi$ and the balanced quasi-shuffle product $\bst$ on $\QB^0$ by a recursive formula on generating series of words (Theorem \ref{stuffle Ybi gen series}, \ref{balanced quasi-shuffle gen series}). Moreover, we will obtain a \emph{regularization map}
\[\reg:(\QB,\bst)\to (\QB^0,\bst),\] 
which also possesses a description on generating series of words (Theorem \ref{reg on gen series}). This regularization allows defining a \emph{regularized coproduct} $\dec^0:\QB^0\to\QB^0\otimes\QB^0$ compatible with the balanced quasi-shuffle product. The lift to generating series of words allows us to show the following.
\begin{thm*}[Theorem \ref{Hopf iso Ybi B}] There is an isomorphism of weight-graded Hopf algebras
\begin{align*} \varphi_{\#}:(\QYbi,\ast,\dec)\to(\QB^0,\bst,\dec^0),
\end{align*}
which satisfies $\varphi_{\#}\circ\swap=\tau\circ\varphi_{\#}$.
\end{thm*} \noindent
The stuffle product $\ast$ as well as the involution $\swap$ on $\QYbi$ can be rephrased in terms of bimoulds (compare to Section \ref{quasi-shuffle and gen series}), this leads to the symmetries usually called \emph{symmetril} and \emph{swap invariant} (those symmetries are extensively studied in \cite{ec}, \cite{sc}). An example of a symmetril and swap invariant bimould is the bimould $\mathfrak{G}=(\mathfrak{G}_d)_{d\geq0}$ of generating series of the combinatorial bi-multiple Eisenstein series (Definition \ref{def G}). On the other hand, the balanced quasi-shuffle product $\bst$ and the involution on $\QB^0$ also lead to symmetries of bimoulds, which we will refer to as \emph{b-symmetril} and \emph{$\tau$-invariant}. Theorem \ref{Hopf iso Ybi B} gives an isomorphism between symmetril and swap invariant bimoulds and b-symmetril and $\tau$-invariant bimoulds. Applying this isomorphism to the bimould $\mathfrak{G}=(\mathfrak{G}_d)_{d\geq0}$ of generating series of the combinatorial bi-multiple Eisenstein series yields the bimould $\mathfrak{B}=(\mathfrak{B}_d)_{d\geq0}$ of generating series of the balanced multiple q-zeta values (Definition \ref{def balanced qMZV}). For example, the b-symmetrility of $\mathfrak{B}$ in depth $2$ reads \\
\scalebox{0.95}{\parbox{.5\linewidth}{%
\begin{align*}
\Gzq{1}{X_1}{Y_1}\cdot \Gzq{1}{X_2}{Y_2}= \Gzq{2}{X_2,X_1}{Y_2,Y_1+Y_2}+\Gzq{2}{X_1,X_2}{Y_1,Y_1+Y_2}+\frac{\Gzq{1}{X_1}{Y_1+Y_2}-\Gzq{1}{X_2}{Y_1+Y_2}}{X_1-X_2}.
\end{align*} }} \\
It combines the product formulas on the generating series of the stuffle and shuffle regularized multiple zeta values given in \eqref{product MZV gen series}. The results in Section \ref{Comparison Ybi B} for the generating series of words allow us to deduce the previously presented properties of the balanced multiple q-zeta values from the b-symmetrility and the $\tau$-invariance of $\mathfrak{B}$.

\noindent \vspace{-0,5cm}\\
\paragraph{\textbf{Structure of the paper.}} We start by recalling Hoffman's quasi-shuffle Hopf algebras (\cite{h}) in Section \ref{quasi-shuffle algebras}. Then in Section \ref{quasi-shuffle and gen series} we develop a general approach relating quasi-shuffle products to symmetries among generating series resp. bimoulds. These general observations will be applied to the quasi-shuffle algebras $(\QYbi,\ast)$ and $(\QB^0,\bst)$ in Section \ref{gen series Ybi section} and \ref{gen series B section}. In Section \ref{gen series B regularization}, we will introduce the regularization and the regularized coproduct for $(\QB^0,\bst)$ and reformulate this in terms of generating series. This allows to obtain the Hopf algebra isomorphism $(\QYbi,\ast,\dec)\to (\QB^0,\bst,\dec^0)$ in Section \ref{Comparison Ybi B}. In Section \ref{qMZV} we recall the definition of the algebra $\Zq$ of multiple q-zeta values (as given in \cite{bk}) and explain shortly its relations to multiple zeta values, quasi-modular forms, and partitions. Then in Section \ref{CMES} we introduce the combinatorial bi-multiple Eisenstein series constructed in \cite{bb}, which form a spanning set of $\Zq$. Applying the results in Section \ref{Comparison Ybi B} to the combinatorial bi-multiple Eisenstein series, we will obtain the balanced multiple q-zeta values and some of its most important properties in Section \ref{balanced qMZV}. Finally, we give some more properties of the balanced multiple q-zeta values in Section \ref{further properties balanced qMZV}.

\noindent \vspace{-0,5cm}\\
\paragraph{\textbf{Acknowledgment.}} This work was mainly part of my PhD thesis. Therefore, I deeply thank my supervisor Ulf Kühn for many helpful comments and discussions on these contents. Moreover, I would like to thank Claudia Alfes-Neumann, Henrik Bachmann, Jan-Willem van Ittersum, and Koji Tasaka for valuable comments on these topics within my PhD project or on earlier versions of this paper.

\section{Quasi-shuffle Hopf algebras} \label{quasi-shuffle algebras}
	
\noindent
We give a short introduction to quasi-shuffle Hopf algebras, which were introduced and studied in \cite{h} and \cite{hi}. They will be used later to describe the product of the algebra of multiple q-zeta values in terms of particular spanning sets.
\vspace{0,2cm} \\
In the following, $R$ is some arbitrary fixed $\QQ$-algebra.
Let $\A$ be an \emph{alphabet}, this means $\A$ is a countable set whose elements are called \emph{letters}. By $R \A$ denote the $R$-module spanned by the letters of $\A$ and let $\nca{R}{\A}$ be the free non-commutative algebra generated by the alphabet $\A$. The monic monomials in $\nca{R}{\A}$ are called \emph{words} with letters in $\A$, and the set of all words is denoted by $\A^*$. Moreover, let $\1$ be the empty word.
\begin{defi} \label{def quasi-shuffle}
Let $\diamond:R \A\times R \A\to R \A$ be an associative and commutative product. Define the \emph{quasi-shuffle product} $\ast_\diamond$ on $\nca{R}{\A}$ recursively by $\1\ast_\diamond w=w\ast_\diamond\1=w$ and 
\begin{align*}
au\ast_\diamond bv=a(u\ast_\diamond bv)+b(au\ast_\diamond v)+(a\diamond b)(u\ast_\diamond v)
\end{align*}
for all $u,v,w\in \nca{R}{\A}$ and $a,b\in \A$.
\end{defi} \noindent
Note that a quasi-shuffle product $\quasish$ can be equally defined recursively from the left and from the right.
\begin{ex} \label{ex quasi-shuffle products}
1. Define \[a\diamond b=0 \text{ for all } a,b\in \A, \]
then we get the well-known \emph{shuffle product}, which is usually denoted by $\shuffle$.
\vspace{0,2cm} \\
2. Consider the bi-alphabet $\Ybi=\{y_{k,m}\mid k\geq1,\ m\geq0\}$ and on $R\Ybi$ define the product 
\begin{align*}
y_{k_1,m_1}\diamond y_{k_2,m_2}= y_{k_1+k_2,m_1+m_2}.
\end{align*}
The obtained quasi-shuffle product is called the \emph{stuffle product} and is denoted by $\ast$. It appears in the context of the combinatorial bi-multiple Eisenstein series (Corollary \ref{quasish cmes}) and is exactly the associated weight-graded product to the quasi-shuffle product of Bachmann's bi-brackets (\cite[Theorem 3.6]{ba}).
\vspace{0,2cm} \\
3. Consider the alphabet $\B=\{b_0,b_1,b_2,\ldots\}$ and define on $R\B$ the product
\[b_i\qdia b_j=\begin{cases} b_{i+j}, &\text{if } i,j\geq 1, \\ 0 &\text{else}. \end{cases}\] 
This quasi-shuffle product occurs for balanced multiple q-zeta values (Theorem \ref{balanced qMZV quasi-shuffle morphism}) and will be called the \emph{balanced quasi-shuffle product}, it is denoted by $\bst$. It is exactly the associated weight-graded product to the quasi-shuffle product of the Schlesinger-Zudilin multiple q-zeta values (\cite[Theorem 3.3]{si}).
\end{ex} \noindent
Let $\dec:\Qnca{\A}\to\Qnca{\A}\otimes\Qnca{\A}$ be the \emph{deconcatenation product}, i.e., we have
\begin{align} \label{deconcatenation coproduct}
\dec(w)=\sum_{uv=w} u\otimes v \quad \text{ for each word } w\in \Qnca{\A}.
\end{align}
We close this section with the following structure theorem of Hoffman.
\begin{thm} (\cite[Theorem 3.2, 3.3]{h}) The tuple $(\nca{R}{\A},\quasish,\dec)$ is an associative, commutative Hopf algebra. There is an isomorphism between all quasi-shuffle Hopf algebras defined over the same alphabet given by some exponential map. \qed 
\end{thm} \noindent
It is well-known that the shuffle algebra $(\Qnca{\A},\shuffle)$ is a free polynomial algebra on the Lyndon words (\cite[Theorem 4.9 (ii)]{re}). The theorem shows that this holds for any quasi-shuffle algebra.

\section{Quasi-shuffle products and generating series} \label{quasi-shuffle and gen series}

\noindent
Inspired by \cite[Section 7]{ih}, we will explain how quasi-shuffle products can be translated into symmetries of commutative generating series resp. bimoulds. This gives a connection between Hoffman's theory of quasi-shuffle Hopf algebras (\cite{h}, \cite{hi}) and Ecalle's theory of bimoulds (\cite{ec}, \cite{sc}). We will use this general setup later to construct a particular nice-behaving model for multiple q-zeta values and explain its properties.
\vspace{0,2cm} \\
Consider some given quasi-shuffle algebra $(\Qnca{\A},\quasish)$. Define the \emph{generic diagonal series} of $\Qnca{\A}$ by
\begin{align*}
\mathcal{W}(\A)=\sum_{w\in\A^*} w\otimes w.
\end{align*}
We want to apply $\QQ$-linear maps to the first factors, usually denoted by $\varphi$, or to the second factors, usually denoted by $\rho$, of $\mathcal{W}(\A)$ to get generating series of different kinds and describe the resulting properties. We begin with an abstract discussion and later a more detailed explanation of special cases is given.
\vspace{0,2cm} \\
Let $\dep:\A^*\to \ZZ_{\geq0}$ be a \emph{depth map} compatible with concatenation, i.e., we have 
\[\dep(uv)=\dep(u)+\dep(v),\qquad u,v\in \A^*.\] 
Denote by $(\A^*)^{(d)}$ the set of all words in $\A^*$ of depth $d$ and by $\Qnca{\A}^{(d)}$ the space spanned by $(\A^*)^{(d)}$. 
\begin{defi} \label{def gen series of words} Let $\rho_{\A}:\Qnca{\A}\to\QQ[Z_1,Z_2,\ldots]$ be a $\QQ$-linear map having the following properties with respect to the depth map:
\begin{itemize}
\item[(i)] There is a strictly increasing sequence $\ell(0)<\ell(1)<\ell(2)<\dots$ of non-negative integers, such that for each $d\geq0$ the restriction of $\rho_{\A}$ to $\Qnca{\A}^{(d)}$ is an injective $\QQ$-linear map
\[\rho_{\A}|_{\Qnca{\A}^{(d)}}:\Qnca{\A}^{(d)}\to \QQ[Z_1,\ldots,Z_{\ell(d)}].\] 
\item[(ii)] We have
\[\rho_{\A}(uv)=\rho_{\A}(u)\rho_{\A}^{[\ell(n)]}(v), \qquad u\in\Qnca{\A}^{(n)},\ v\in \Qnca{\A},\]
where $\rho_{\A}^{[n]}$ denotes the $\QQ$-linear map obtained from $\rho_{\A}$ by shifting the variables $Z_i$ to $Z_{n+i}$, so $\rho_{\A}^{[n]}\Big(\Qnca{\A}^{(d)}\Big)\subset\QQ[Z_{n+1},\ldots,Z_{n+\ell(d)}]$.
\end{itemize}
\end{defi} 
\begin{defi} \label{def Gen}
Let $\rho_{\A}:\Qnca{\A}\to\QQ[Z_1,Z_2,\ldots]$ be a $\QQ$-linear map as in Definition \ref{def gen series of words}. The (commutative) \emph{generating series of words} in $\Qnca{\A}$ associated to $\rho_{\A}$ are given by
\[\Gen_d(Z_1,\ldots,Z_{\ell(d)})=\sum_{w\in (\A^*)^{(d)}} w\rho(w)\in \Qnca{\A}\llbracket Z_1,\ldots,Z_{\ell(d)}\rrbracket , \qquad d\geq0.\]
\end{defi} \noindent
We will often drop the index $d$ and simply write $\Gen(Z_1,\ldots,Z_{\ell(d)})$.
\begin{ex} Consider the alphabet $\Y=\{y_1,y_2,y_3,\ldots\}$ and define the depth of a word in $\QY$ by
\[\dep(y_{k_1}\dots y_{k_d})=d.\]
The map 
\begin{align*}
\rho_{\Y}:\QY&\to \QQ[Z_1,Z_2,\ldots], \\
y_{k_1}\dots y_{k_d}&\mapsto Z_1^{k_1-1}\dots Z_d^{k_d-1}
\end{align*}
is a $\QQ$-linear map satisfying the conditions in Definition \ref{def gen series of words} with $\ell(d)=d$. The associated generating series of words are given by $\GenY_0=\1$ and 
\[\GenY_d(Z_1,\ldots,Z_d)=\sum_{k_1,\ldots,k_d\geq1} y_{k_1}\dots y_{k_d}Z_1^{k_1-1}\dots Z_d^{k_d-1},\qquad d\geq1.\]
\end{ex} \noindent
In the following sections, we will compute more involved examples.
\begin{prop} \label{circle}
Let $\rho_{\A}:\Qnca{A}\to\QQ[Z_1,Z_2,\ldots]$ be a $\QQ$-linear map as in Definition \ref{def gen series of words} with $\ell(d_1)+\ell(d_2)=\ell(d_1+d_2)$ for all $d_1,d_2\geq0$. Then the $\QQ\llbracket Z_1,Z_2,\ldots\rrbracket $-linear extension of the concatenation product $\cdot$ satisfies
\[\Gen_n(Z_1,\ldots,Z_{\ell(n)})\cdot \Gen_{d-n}(Z_{\ell(n)+1},\ldots,Z_{\ell(d)})=\Gen_d(Z_1,\ldots,Z_{\ell(d)})\]
for all $0\leq n\leq d$.
\end{prop}
\begin{proof} For $n=0,d$ the formula is obvious. For $0<n<d$, compute
\begin{align*}
&\Gen_n(Z_1,\ldots,Z_{\ell(n)})\cdot \Gen_{d-n}(Z_{\ell(n)+1},\ldots,Z_{\ell(d)})=\sum_{u\in (\A^*)^{(n)}}\sum_{v\in (\A^*)^{(d-n)}} \hspace{-0,2cm}uv \rho_{\A}(u)\rho_{\A}^{[\ell(n)]}(v)\\
&=\sum_{u\in (\A^*)^{(n)}}\sum_{v\in (\A^*)^{(d-n)}} uv \rho_{\A}(uv)=\sum_{w\in (\A^*)^{(d)}} w\rho_{\A}(w)=\Gen_d(Z_1,\ldots,Z_{\ell(d)}).
\end{align*}
\vspace{-0,7cm} \\ \end{proof} \noindent
If we extend the quasi-shuffle product $\quasish$ on $\Qnca{\A}$ to $\Qnca{\A}\llbracket Z_1,Z_2,\ldots\rrbracket $ by $\QQ\llbracket Z_1,Z_2,\ldots\rrbracket $-linearity, then by definition of the quasi-shuffle product $\quasish$ we have for all $0\leq n\leq d$ that
\[\Gen_n(Z_1,\ldots,Z_{\ell(n)})\quasish \Gen_{d-n}(Z_{\ell(n)+1},\ldots,Z_{\ell(d)})\in\Qnca{\A}\llbracket Z_1,\ldots,Z_{\ell(d)}\rrbracket .\]
For some very well-behaved quasi-shuffle products appearing for example in the theory of multiple zeta values and multiple q-zeta values, it is possible to describe this product on the generating series of words by a recursive explicit formula with respect to concatenation. This will explain the origin of a particular kind of symmetries occurring in Ecalle's theory of bimoulds (\cite{ec}). Some important examples related to multiple q-zeta values will be discussed in the following sections, for some other examples also related to multiple zeta values, we refer to \cite[Appendix A.4 + A.5]{bu}.
\vspace{0,2cm} \\
To translate the quasi-shuffle products into symmetries among bimoulds, we apply a $\QQ$-linear map  $\varphi:\Qnca{\A}\to R$ into some $\QQ$-algebra $R$ to the first component of such a generating series of words $\Gen$. In other words, we consider the image of the generic diagonal series $\mathcal{W}(\A)$ under $\varphi\otimes\rho_{\A}$.
\begin{defi} \label{def rho varphi gen series}
Let $\rho_{\A}:\Qnca{\A}\to \QQ[Z_1,Z_2,\ldots]$ be a $\QQ$-linear map as in Definition \ref{def gen series of words}, $R$ be a $\QQ$-algebra and $\varphi:\Qnca{\A}\to R$ be a $\QQ$-linear map. Then the \emph{(commutative) generating series with coefficients in $R$} associated to $(\varphi,\rho_{\A})$ are given by
\[\GenR_d(Z_1,\ldots,Z_{\ell(d)})=\sum_{w\in (\A^*)^{(d)}} \varphi(w)\rho_{\A}(w)\in R\llbracket Z_1,\ldots,Z_{\ell(d)}\rrbracket ,\qquad d\geq0.\]
\end{defi} \noindent
As before, we will often drop the depth index and simply write $\GenR(Z_1,\ldots,Z_{\ell(d)})$. \\
Such kind of sequences $\GenR=\Big(\GenR_d\Big)_{d\geq0}$ will occur as examples of Ecalle's (bi-)moulds (\cite{ec}). 
\begin{defi} Let $R$ be a $\QQ$-algebra. A sequence \\
\scalebox{0.9}{\parbox{.5\linewidth}{%
\[M=\left(M_d(X_1,\ldots,X_d)\right)_{d\geq0}=\left(M_0(\emptyset),M_1(X_1),M_2(X_1,X_2),\ldots\right)\in\prod_{d\geq0}R\llbracket X_1,\ldots,X_d\rrbracket \] }} \\
is called a \emph{mould} with coefficients in $R$. Similarly, a sequence \\
\scalebox{0.9}{\parbox{.5\linewidth}{%
\[M=\left(M_d\bi{X_1,\ldots,X_d}{Y_1,\ldots,Y_d}\right)_{d\geq0}=\left(M_0(\emptyset),M_1\bi{X_1}{Y_1},M_2\bi{X_1,X_2}{Y_1,Y_2},\ldots\right)\in\prod_{d\geq0}R\llbracket X_1,Y_1,\ldots,X_d,Y_d\rrbracket \] }} \\
is called a \emph{bimould} with coefficients in $R$.
\end{defi} \noindent
If $\rho_{\A}:\Qnca{\A}\to \QQ[X_1,X_2,\ldots]$ is a map as in Definition \ref{def gen series of words} with $\ell(d)=d$ for all $d\geq0$, then the corresponding family of generating series of words $\Gen=(\Gen_d)_{d\geq0}$ is a mould with coefficients in $\Qnca{\A}$. For any $\QQ$-linear map $\varphi:\Qnca{\A}\to R$, the family of generating series $\GenR=(\GenR_d)_{d\geq0}$ associated to $(\varphi,\rho_{\A})$ is a mould with coefficients in $R$. Similarly, if $\rho_{\A}:\Qnca{\A}\to \QQ[X_1,Y_1,X_2,Y_2,\ldots]$ is a map as in Definition \ref{def gen series of words} with $\ell(d)=2d$ for every $d\geq0$, then the corresponding family of generating series of words $\Gen=(\Gen_d)_{d\geq0}$ is a bimould with coefficients in the algebra $\Qnca{\A}$, and also the family of generating series $\GenR=(\GenR_d)_{d\geq0}$ associated to $(\varphi,\rho_{\A})$ for each $\QQ$-linear map $\varphi:\Qnca{\A}\to R$ is a bimould with coefficients in $R$.
\begin{defi} \label{varphi,rho-symmetric}
Let $R$ be a $\QQ$-algebra and $\ell(0)<\ell(1)<\ell(2)<\dots$ a sequence of non-negative integers. A sequence $M=(M_d)_{d\geq0}$ in $\prod_{d\geq0} R\llbracket Z_1,\ldots,Z_{\ell(d)}\rrbracket $ is called \emph{$(\varphi_{\quasish},\rho_{\A})$-symmetric} if there exists a $\QQ$-algebra morphism $\varphi_{\quasish}:(\Qnca{\A},\quasish)\to R$ and a $\QQ$-linear map $\rho_{\A}:\Qnca{\A}\to\QQ\llbracket Z_1,Z_2,\ldots\rrbracket $ satisfying the conditions in Definition \ref{def gen series of words}, such that for all $d\geq0$
\[M_d=\GenRqs_d.\]
\end{defi} \noindent
In the following sections, we will consider two particular quasi-shuffle algebras, which will give rise to the notion of symmetril and b-symmetril bimoulds following Definition \ref{varphi,rho-symmetric}.
\vspace{0,2cm} \\
Let $M=(M_d)_{d\geq0}\in \prod_{d\geq0} R\llbracket Z_1,\ldots,Z_d\rrbracket $ be such a $(\varphi_{\quasish},\rho_{\A})$-symmetric sequence. Then, one obtains immediately from the definition that for $0<n<d$
\begin{align} \label{formula phi,rho sym}
&M_n(Z_1,\ldots,Z_{\ell(n)})M_{d-n}(Z_{\ell(n)+1},\ldots,Z_{\ell(d)})\\
&\hspace{5cm}=\varphi_{\quasish}\Big(\Gen_n(Z_1,\ldots,Z_{\ell(n)})\quasish \Gen_{d-n}(Z_{\ell(n)+1},\ldots,Z_{\ell(d)})\Big). \nonumber
\end{align}
The right-hand side is an element in $\Qnca{\A}\llbracket Z_1,\ldots,Z_{\ell(d)}\rrbracket $, so the map $\varphi_{\quasish}:\Qnca{\A}\to R$ needs to be extended by $\QQ\llbracket Z_1,Z_2,\ldots\rrbracket $-linearity. As mentioned before, in some special cases the right-hand side can be described in terms of an explicit recursive formula.
\vspace{0,2cm} \\
Alternatively, we can first apply the evaluation map $\varphi:\Qnca{\A}\to R$ for some $\QQ$-algebra $R$ to the generic diagonal series $\mathcal{W}(\A)$. 
\begin{defi} \label{non-com gen series for varphi}
Let $R$ be a $\QQ$-algebra and $\varphi:\Qnca{\A}\to R$ be a $\QQ$-linear map. Define the \emph{(non-commutative) generating series with coefficients in $R$} associated to $\varphi$ by 
\[\Genphi=\sum_{w\in \A^*} \varphi(w)w\in \ncac{R}{\A}.\]
Here $\ncac{R}{\A}$ denotes the non-commutative algebra of power series over $R$ generated by $\A$.
\end{defi} \noindent
The generating series $\Genphi$ can also be decomposed into its homogeneous depth components $\Genphi_d$ for $d\geq0$ (similar to Definition \ref{def Gen}). 
\begin{prop} Let $R$ be a $\QQ$-algebra and $\varphi:\Qnca{\A}\to R$ be a $\QQ$-linear map. Assume that the algebra $(\Qnca{\A},\quasish)$ is graded with $\operatorname{deg}(a)\geq1$ for all $a\in\A$, and denote by $\Delta_{\quasish}$ the dual completed coproduct to $\quasish$. 
\vspace{0,1cm}\\
The map $\varphi$ is an algebra morphism for the quasi-shuffle product $\quasish$ if and only if $\Genphi$ is grouplike for the coproduct $\Delta_{\quasish}$.
\end{prop}
\begin{proof} By duality, we have \[\Delta_{\quasish}\Big(\Genphi\Big)=\Delta_{\quasish}\left(\sum_{w\in \A^*}\varphi(w)w\right)=\sum_{u,v\in\A^*} \varphi(u\quasish v) u\otimes v.\] 
So $\varphi$ is an algebra morphism for the quasi-shuffle product $\quasish$ if and only if
\begin{align*} 
\Delta_{\quasish}\Big(\Genphi\Big)=\sum_{u,v\in\A^*} \varphi(u)\varphi(v) u\otimes v=\Genphi\otimes\Genphi.
\end{align*}
\vspace{-0,7cm} \\ \end{proof} \noindent
In particular, applying the map $\varphi$ involves always a dualization process. Summarizing this section, we obtain the following diagram \\
\[\begin{tikzcd}
& \sum\limits_{w\in\A^*} w\otimes w \arrow[dr,"1\otimes\rho"] \arrow[dl,"\varphi\otimes1"'] & \\
\sum\limits_{w\in \A^*} \varphi(w)w
\arrow[dr,"1\otimes\rho"] && \left(\sum\limits_{w\in (\A^*)^{(d)}} w\rho(w)\right)_{d\geq0}
\arrow[dl,"\varphi\otimes1"] \\ 
&\left(\sum\limits_{w\in (\A^*)^{(d)}} \varphi(w)\rho(w)\right)_{d\geq0}
\end{tikzcd}.\]

\section{Generating series of words in $\QYbi$} \label{gen series Ybi section}

\noindent
Consider the alphabet $\Ybi=\{y_{k,m}\mid k\geq1,\ m\geq0\}$ and let the \emph{weight} and the \emph{depth} of a word in $\QYbi$ be given by 
\[\wt(y_{k_1,m_1}\dots y_{k_d,m_d})=k_1+\dots+k_d+m_1+\dots+m_d,\qquad \dep(y_{k_1,m_1}\dots y_{k_d,m_d})=d,\] 
for all $k_1,\ldots,k_d\geq1,\ m_1,\ldots,m_d\geq0$. Define the $\QQ$-linear map
\begin{align} \label{def rho Y^bi}
\rho_{\Ybi}:\QYbi&\to \QQ[X_1,Y_1,X_2,Y_2,\ldots], \\
y_{k_1,m_1}\dots y_{k_d,m_d}&\mapsto X_1^{k_1-1}\frac{Y_1^{m_1}}{m_1!}\dots X_d^{k_d-1}\frac{Y_d^{m_d}}{m_d!},
\end{align}
then $\rho_{\Ybi}$ satisfies the conditions in Definition \ref{def gen series of words} with $\ell(d)=2d$ for all $d\geq0$. The generating series of words in $\QYbi$ associated to $\rho_{\Ybi}$ are given by $\GenYbi_0=\1$ and for $d\geq 1$ by 
\begin{align} \label{gen series Ybi}
\GenYbi_d\bi{X_1,\ldots,X_d}{Y_1,\ldots,Y_d}=\sum_{\substack{k_1,\ldots,k_d\geq1 \\ m_1,\ldots,m_d\geq0}} y_{k_1,m_1}\dots y_{k_d,m_d}X_1^{k_1-1}\frac{Y_1^{m_1}}{m_1!}\dots X_d^{k_d-1}\frac{Y_d^{m_d}}{m_d!}.
\end{align}
The stuffle product $\ast$ on $\QYbi$ is one of the well-behaved quasi-shuffle products, where we can give an explicit recursive formula on the level of generating series of words.
\begin{thm} \label{stuffle Ybi gen series}
For the stuffle product $\ast$ on $\QYbi$ (Example \ref{ex quasi-shuffle products}.2.), one obtains for all $0<n<d$ that $\1\ast \GenYbi_n=\GenYbi_n\ast \1=\GenYbi_n$ and \\
\scalebox{0.95}{\parbox{.5\linewidth}{%
\begin{align*} 
&\GenYbi\bi{X_1,\ldots,X_n}{Y_1,\ldots,Y_n}\ast \GenYbi\bi{X_{n+1},\ldots,X_d}{Y_{n+1},\ldots,Y_d}\\
&=\GenYbi\bi{X_1}{Y_1}\cdot\left(\GenYbi\bi{X_2,\ldots,X_n}{Y_2,\ldots,Y_n}\ast \GenYbi\bi{X_{n+1},\ldots,X_d}{Y_{n+1},\ldots,Y_d} \right)\\
&+\GenYbi\bi{X_{n+1}}{Y_{n+1}}\cdot\left(\GenYbi\bi{X_1,\ldots,X_n}{Y_1,\ldots,Y_n}\ast \GenYbi\bi{X_{n+2},\ldots,X_d}{Y_{n+2},\ldots,Y_d}\right) \\ &+\frac{\GenYbi\bi{X_1}{Y_1+Y_{n+1}}-\GenYbi\bi{X_{n+1}}{Y_1+Y_{n+1}}}{X_1-X_{n+1}}\cdot \left(\GenYbi\bi{X_2,\ldots,X_n}{Y_2,\ldots,Y_n}\ast \GenYbi\bi{X_{n+2},\ldots,X_d}{Y_{n+2},\ldots,Y_d}\right).
\end{align*} }}
\end{thm}
\begin{proof} First, consider the case $d=2$ and compute directly
\begin{align*}
&\GenYbi\bi{X_1}{Y_1}\ast\GenYbi\bi{X_2}{Y_2}=\sum_{\substack{k_1,k_2\geq1 \\ m_1,m_2\geq0}}(y_{k_1,m_1}\ast y_{k_2,m_2})X_1^{k_1-1}\frac{Y_1^{m_1}}{m_1!}X_2^{k_2-1}\frac{Y_2^{m_2}}{m_2!}\\
&=\sum_{\substack{k_1,k_2\geq1 \\ m_1,m_2\geq0}} \Big(y_{k_1,m_1}y_{k_2,m_2}+y_{k_2,m_2}y_{k_1,m_1}+y_{k_1+k_2,m_1+m_2}\Big) X_1^{k_1-1}\frac{Y_1^{m_1}}{m_1!}X_2^{k_2-1}\frac{Y_2^{m_2}}{m_2!} \\
&=\GenYbi\bi{X_1,X_2}{Y_1,Y_2}+\GenYbi\bi{X_2,X_1}{Y_2,Y_1}+\frac{\GenYbi\bi{X_1}{Y_1+Y_2}-\GenYbi\bi{X_2}{Y_1+Y_2}}{X_1-X_2},
\end{align*}
where the last step follows from simple power series manipulations. Since the definition of the stuffle product $\ast$ as well as the above formula for the generating series $\GenYbi$ is recursive, the arbitrary depth case follows similarly by induction.
\end{proof} \noindent
Following Definition \ref{varphi,rho-symmetric}, the quasi-shuffle algebra $(\QYbi,\ast)$ gives rise to a symmetry among bimoulds, which was studied extensively in \cite{ec} and \cite{sc}.
\begin{defi} \label{def symmetril}
A bimould $M=(M_d)_{d\geq0}\in \prod_{d\geq0}R\llbracket X_1,Y_1,\ldots,X_d,Y_d\rrbracket $ with coefficients in some $\QQ$-algebra $R$ is \emph{symmetril} if there is an algebra morphism $\varphi_\ast:(\QYbi,\ast)\to R$, such that $M$ is $(\rho_{\Ybi},\varphi_\ast)$-symmetric, i.e., we have $M_0=1$ and $M_d$ is for $d\geq 1$ given by\begin{flushleft}
	
\end{flushleft}
\[M_d\bi{X_1,\ldots,X_d}{Y_1,\ldots,Y_d}=\sum_{\substack{k_1,\ldots,k_d\geq1 \\ m_1,\ldots,m_d\geq0}} \varphi_\ast(y_{k_1,m_1}\dots y_{k_d,m_d})X_1^{k_1-1}\frac{Y_1^{m_1}}{m_1!}\dots X_d^{k_d-1}\frac{Y_d^{m_d}}{m_d!}.\]
We refer to $\varphi_\ast$ as the \emph{coefficient map} of $M$. 
\end{defi} 
\begin{ex} \label{ex symmetril}
Applying the recursive formula for $\ast$ on generating series of words given in Theorem \ref{stuffle Ybi gen series}, one obtains that Definition \ref{def symmetril} is equivalent to the symmetrility defined in \cite[p. 17ff]{sc}. For example, symmetrility in depths $2$ and $3$ means 
\begin{align*}
M\bi{X_1}{Y_1}\cdot M\bi{X_2}{Y_2} =&\ M\bi{X_1,X_2}{Y_1,Y_2}+M\bi{X_2,X_1}{Y_2,Y_1}+\frac{M\bi{X_1}{Y_1+Y_2}-M\bi{X_2}{Y_1+Y_2}}{X_1-X_2}, \\
M\bi{X_1}{Y_1}\cdot M\bi{X_2,X_3}{Y_2,Y_3} = & \ M\bi{X_1,X_2,X_3}{Y_1,Y_2,Y_3}+M\bi{X_2,X_1,X_3}{Y_2,Y_1,Y_3}+M\bi{X_2,X_3,X_1}{Y_2,Y_3,Y_1} \\
&+\frac{M\bi{X_1,X_3}{Y_1+Y_2,Y_3}-M\bi{X_2,X_3}{Y_1+Y_2,Y_3}}{X_1-X_2}+\frac{M\bi{X_2,X_1}{Y_2, Y_1+Y_3}-M\bi{X_2,X_3}{Y_2,Y_1+Y_3}}{X_1-X_3}.
\end{align*}
\end{ex} \noindent
The deconcatenation coproduct on $\QYbi$ (see \eqref{deconcatenation coproduct}) is compatible with the translation map $\rho_{\Ybi}$, a straight-forward computation gives the following.
\begin{prop} \label{dec gen series Ybi}
For each $d\geq0$, we have
\begin{align*}
\dec\GenYbi\bi{X_1,\ldots,X_d}{Y_1,\ldots,Y_d}=\sum_{i=0}^d \GenYbi\bi{X_1,\ldots,X_i}{Y_1,\ldots,Y_i}\otimes\GenYbi\bi{X_{i+1},\ldots,Y_d}{X_{i+1},\ldots,Y_d}.
\end{align*}
\end{prop}

\section{Generating series of words in $\QB^0$} \label{gen series B section}

\noindent
Consider the alphabet $\B=\{b_0,b_1,b_2,\ldots\}$ and define the \emph{weight} and the \emph{depth} of a word in $\QB$ by 
\[\wt(b_{s_1}\dots b_{s_l})=s_1+\dots+s_l+\#\{i\mid s_i=0\},\qquad \dep(b_{s_1}\dots b_{s_l})=l-\#\{i\mid s_i=0\},\]
for all $s_1,\ldots,s_l\geq0$. Let $\rho_\B$ be the $\QQ$-linear map given by
\begin{align} \label{def rho B}
\rho_\B:\QB&\to\QQ[Y_0,X_1,Y_1,X_2,Y_2,\ldots], \\
b_0^{m_0}b_{k_1}b_0^{m_1}\dots b_{k_d}b_0^{m_d}&\mapsto Y_0^{m_0}X_1^{k_1-1}Y_1^{m_1}\dots X_d^{k_d-1}Y_d^{m_d} \quad (k_1,\ldots,k_d\geq1,m_1,\ldots,m_d\geq0),
\end{align}
then $\rho_\B$ satisfies the conditions in Definition \ref{def gen series of words} with $\ell(d)=2d+1$ for all $d\geq0$. The generating series of words associated to $\rho_\B$ are given by for $d\geq0$
\[\GenB_d\bi{\hspace{0,5cm}X_1,\ldots,X_d}{Y_0;Y_1,\ldots,Y_d}=\sum_{\substack{k_1,\ldots,k_d\geq1 \\ m_0,\ldots,m_d\geq0}} b_0^{m_0}b_{k_1}b_0^{m_1}\dots b_{k_d}b_0^{m_d}Y_0^{m_0}X_1^{k_1-1}Y_1^{m_1}\dots X_d^{k_d-1}Y_d^{m_d}.\]
Let $\QB^0$ be the subspace of $\QB$ generated by all words not starting in $b_0$, i.e., we have
\[\QB^0=\QB\backslash b_0\QB.\] 
Define the corresponding diagonal series
\[\mathcal{W}(\B)^0=\sum_{w\in \B^*\cap \QB^0} w\otimes w.\] 
Let $\rho_{\B}^0:\QB^0\to \QQ[X_1,Y_1,X_2,Y_2,\ldots]$ be the restriction of $\rho_{\B}$ to $\QB^0$. Then the generating series of words in $\QB^0$ associated to $\rho_{\B}^0$ are given by 
$\GenBo_0=\1$ and
\[\GenBo_d\bi{X_1,\ldots,X_d}{Y_1,\ldots,Y_d}=\sum_{\substack{k_1,\ldots,k_d\geq1 \\ m_1,\ldots,m_d\geq0}} b_{k_1}b_0^{m_1}\dots b_{k_d}b_0^{m_d}X_1^{k_1-1}Y_1^{m_1}\dots X_d^{k_d-1}Y_d^{m_d},\quad d\geq1.\] 
On $\QB^0$ the balanced quasi-shuffle product $\bst$ is very well-behaved, thus we can also give a recursive definition on generating series of words for this product.
\begin{thm} \label{balanced quasi-shuffle gen series}
For the balanced quasi-shuffle product $\bst$ (Example \ref{ex quasi-shuffle products}.3.), one obtains for all $0<n<d$ that $\1\bst \GenBo_n=\GenBo_n\bst\ \1=\GenBo_n$ and \\
\scalebox{0.95}{\parbox{.5\linewidth}{%
\begin{align*}
&\GenBo\bi{X_1,\ldots,X_n}{Y_1,\ldots,Y_n}\bst \GenBo\bi{X_{n+1},\ldots,X_d}{Y_{n+1},\ldots,Y_d} \\
&=\left(\GenBo\bi{X_1,\ldots,X_{n-1}}{Y_1,\ldots,Y_{n-1}}\bst\GenBo\bi{X_{n+1},\ldots,X_d}{Y_{n+1},\ldots,Y_d}\right) \cdot \GenBo\bi{X_n}{Y_n+Y_d} \\
&+\left(\GenBo\bi{X_1,\ldots,X_n}{Y_1,\ldots,Y_n}\bst\GenBo\bi{X_{n+1},\ldots,X_{d-1}}{Y_{n+1},\ldots,Y_{d-1}}\right)\cdot\GenBo\bi{X_d}{Y_n+Y_d} \\
&+\left(\GenBo\bi{X_1,\ldots,X_{n-1}}{Y_1,\ldots,Y_{n-1}}\bst\GenBo\bi{X_{n+1},\ldots,X_{d-1}}{Y_{n+1},\ldots,Y_{d-1}}\right)\cdot\frac{\GenBo\bi{X_n}{Y_n+Y_d}-\GenBo\bi{X_d}{Y_n+Y_d}}{X_n-X_d}.
\end{align*} }}
\end{thm} 
\begin{proof} Consider the case $d=2$. We obtain \\
\scalebox{0.95}{\parbox{.5\linewidth}{%
\begin{align} \label{9}
\GenBo\bi{X_1}{Y_1}&=\left(\sum_{k\geq1} b_kX_1^{k-1}\right)\cdot\frac{\1}{\1-b_0Y_1}\\
&=\sum_{k\geq1} b_kX_1^{k-1}+\GenBo\bi{X_1}{Y_1}\cdot b_0Y_1.
\end{align} }} \\
By applying the second identity in \eqref{9} and the definition of the balanced quasi-shuffle product (we use the recursive definition from the right), we obtain \\
\scalebox{0.95}{\parbox{.5\linewidth}{%
\begin{align*}
&\GenBo\bi{X_1}{Y_1}\bst\GenBo\bi{X_2}{Y_2}\\
&=\left(\sum_{k\geq1} b_kX_1^{k-1}\right) \bst\left(\sum_{k\geq1} b_kX_2^{k-1}\right)+\left(\sum_{k\geq1} b_kX_1^{k-1}\right)\bst\left(\GenBo\bi{X_2}{Y_2}\cdot b_0Y_2\right) \hspace{3cm}\\
&+ \left(\GenBo\bi{X_1}{Y_1}\cdot b_0Y_1\right)\bst\left(\sum_{k\geq1} b_kX_2^{k-1}\right)+\left(\GenBo\bi{X_1}{Y_1}\cdot b_0Y_1\right)\bst\left(\GenBo\bi{X_2}{Y_2}\cdot b_0Y_2\right) 
\end{align*} }} \\
\scalebox{0.95}{\parbox{.5\linewidth}{%
\begin{align*}
&=\left(\sum_{k\geq1} b_kX_1^{k-1}\right) \bst\left(\sum_{k\geq1} b_kX_2^{k-1}\right)+\left(\left(\sum_{k\geq1} b_kX_1^{k-1}\right)\bst\left(\GenBo\bi{X_2}{Y_2}Y_2\right)\right)\cdot b_0\\
&+\GenBo\bi{X_2}{Y_2}\cdot b_0Y_2\cdot \sum_{k\geq1} b_kX_1^{k-1}+\left(\left(\GenBo\bi{X_1}{Y_1}Y_1\right)\bst\left(\sum_{k\geq1} b_kX_2^{k-1}\right)\right)\cdot b_0 \\
&+\GenBo\bi{X_1}{Y_1}\cdot b_0Y_1\cdot \sum_{k\geq1} b_kX_2^{k-1}+\left(\left(\GenBo\bi{X_1}{Y_1}Y_1\right)\bst\left(\GenBo\bi{X_2}{Y_2}\cdot b_0Y_2\right)\right)\cdot b_0 \\
&+\left(\left(\GenBo\bi{X_1}{Y_1}\cdot b_0Y_1\right)\bst\left(\GenBo\bi{X_2}{Y_2}Y_2\right) \right)\cdot b_0.
\end{align*} }} \\
Applying again the second identity in \eqref{9} together with some cancellation, we deduce \\
\scalebox{0.935}{\parbox{.5\linewidth}{%
\begin{align*}
&\GenBo\bi{X_1}{Y_1}\bst\GenBo\bi{X_2}{Y_2}\\
&=\left(\sum_{k\geq1} b_kX_1^{k-1}\right)\bst\left(\sum_{k\geq1} b_kX_2^{k-1}\right)-2\left(\sum_{k\geq1} b_kX_1^{k-1}\right)\cdot\left(\sum_{k\geq1} b_kX_2^{k-1}\right)+\GenBo\bi{X_2}{Y_2}\cdot\sum_{k\geq1} b_kX_1^{k-1} \\
&+\GenBo\bi{X_1}{Y_1}\cdot\sum_{k\geq1} b_kX_2^{k-1}+\left(\GenBo\bi{X_1}{Y_1}\bst\GenBo\bi{X_2}{Y_2}\right)\cdot b_0(Y_1+Y_2).
\end{align*} }} \\
Using the first identity in \eqref{9} together with the formula for the concatenation product in terms of generating series of words (Proposition \ref{circle}), we obtain \\
\scalebox{0.95}{\parbox{.5\linewidth}{%
\begin{align} \label{formula34}
&\GenBo\bi{X_1}{Y_1}\bst\GenBo\bi{X_2}{Y_2}\\
&=\left(\left(\sum_{k\geq1} b_kX_1^{k-1}\right)\bst\left(\sum_{k\geq1} b_kX_2^{k-1}\right)-2\left(\sum_{k\geq1} b_kX_1^{k-1}\right)\cdot\left(\sum_{k\geq1} b_kX_2^{k-1}\right)\right)\cdot\frac{\1}{\1-b_0(Y_1+Y_2)} \nonumber \\
&\hspace{0,45cm}+\GenBo\bi{X_2,X_1}{Y_2,Y_1+Y_2}+\GenBo\bi{X_1,X_2}{Y_1,Y_1+Y_2}. \nonumber
\end{align} }} \\
Moreover, one verifies by applying the definition of $\bst$ and some power series manipulation \\
\scalebox{0.92}{\parbox{.5\linewidth}{%
\begin{align*}
&\left(\sum_{k\geq1} b_kX_1^{k-1}\right)\bst\left(\sum_{k\geq1} b_kX_2^{k-1}\right)=2\left(\sum_{k\geq1} b_kX_1^{k-1}\right)\cdot\left(\sum_{k\geq1} b_kX_2^{k-1}\right)+\frac{\sum_{k\geq1} b_kX_1^{k_1}-\sum_{k\geq1} b_k X_2^{k-1}}{X_1-X_2}.
\end{align*} }} \\
Applying this result to \eqref{formula34} together with \eqref{9} gives \\
\scalebox{0.95}{\parbox{.5\linewidth}{%
\begin{align*}
&\GenBo\bi{X_1}{Y_1}\bst\GenBo\bi{X_2}{Y_2}\\
&\hspace{1,8cm}=\GenBo\bi{X_2,X_1}{Y_2,Y_1+Y_2}+\GenBo\bi{X_1,X_2}{Y_1,Y_1+Y_2}+\frac{\GenBo\bi{X_1}{Y_1+Y_2}-\GenBo\bi{X_2}{Y_1+Y_2}}{X_1-X_2}.
\end{align*} }} \\
This equals exactly the claimed formula of $\bst$ for the generating series in depth $2$. Since the definition of the product $\bst$ and the above generating series formula are both recursive, we obtain the desired formula in arbitrary depth by applying induction and the same computations as before.
\end{proof} \noindent
The quasi-shuffle algebra $(\QB^0,\bst)$ gives rise to symmetry among bimoulds according to Definition \ref{varphi,rho-symmetric}.
\begin{defi} \label{def b-symmetril}
A bimould $M=(M_d)_{d\geq0}\in \prod_{d\geq0} R\llbracket X_1,Y_1,\ldots,X_d,Y_d\rrbracket $ with coefficients in some $\QQ$-algebra $R$ is \emph{b-symmetril} if there is an algebra morphism $\varphi_{\bst}:(\QB^0,\bst)\to R$, such that $M$ is $(\varphi_{\bst},\rho_\B^0)$-symmetric, i.e., we have $M_0=1$ and for all $d\geq1$
\[M_d\bi{X_1,\ldots,X_d}{Y_1,\ldots,Y_d}=\sum_{\substack{k_1,\ldots,k_d\geq1 \\ m_1,\ldots,m_d\geq0}} \varphi_{\bst}(b_{k_1}b_0^{m_1}\dots b_{k_d}b_0^{m_d})X_1^{k_1-1}Y_1^{m_1}\dots X_d^{k_d-1}Y_d^{m_d}.\]	
As before, we call the map $\varphi_{\bst}$ the \emph{coefficient map} of $M$. 
\end{defi}
\begin{ex} \label{ex b-symmetril} Applying the recursive formula for $\bst$ on the generating series of words given in Theorem \ref{balanced quasi-shuffle gen series}, we obtain that b-symmetrility in depth $2$ and $3$ means
\begin{align*}
M\bi{X_1}{Y_1}\cdot M\bi{X_2}{Y_2} =&\ M\bi{X_2,X_1}{Y_2,Y_1+Y_2}+M\bi{X_1,X_2}{Y_1,Y_1+Y_2}+\frac{M\bi{X_1}{Y_1+Y_2}-M\bi{X_2}{Y_1+Y_2}}{X_1-X_2}, \\
M\bi{X_1}{Y_1}\cdot M\bi{X_2,X_3}{Y_2,Y_3} =&\ M\bi{X_2,X_3,X_1}{Y_2,Y_3,Y_1+Y_3}+M\bi{X_2,X_1,X_3}{Y_2,Y_1+Y_2,Y_1+Y_3}\\
&+ M\bi{X_1,X_2,X_3}{Y_1,Y_1+Y_2,Y_1+Y_3}+\frac{M\bi{X_2,X_1}{Y_2, Y_1+Y_3}-M\bi{X_2,X_3}{Y_2,Y_1+Y_3}}{X_1-X_3}\\
&+\frac{M\bi{X_1,X_3}{Y_1+Y_2,Y_1+Y_3}-M\bi{X_2,X_3}{Y_1+Y_2,Y_1+Y_3}}{X_1-X_2}.
\end{align*}
\end{ex} 

\section{Generating series of words in $\QB$ and regularization} \label{gen series B regularization}

\noindent
The deconcatenation coproduct on $\QB$ (defined in \eqref{deconcatenation coproduct}) is compatible with the translation map $\rho_{\B}$, hence it can be described in terms of the generating series $\GenB$ of words in $\QB$. But the deconcatenation coproduct does not preserve the subspace $\QB^0$. Therefore, we will introduce a regularization map and a regularized coproduct and describe both in terms of the generating series $\GenBo$ of words in $\QB^0$.
\begin{prop} \label{dec on GenB}
For all $d\geq0$, we have
\[\dec\GenB\bi{\hspace{0,5cm}X_1,\ldots,X_d}{Y_0;Y_1,\ldots,Y_d}=\sum_{i=0}^d \GenB\bi{\hspace{0,5cm} X_1,\ldots,X_i}{Y_0;Y_1,\ldots,Y_i}\otimes\GenB\bi{\hspace{0,5cm} X_{i+1},\ldots,X_d}{Y_i;Y_{i+1},\ldots,Y_d}.\]
\end{prop}
\begin{proof} By definition of the deconcatenation coproduct, we have \\
\scalebox{0.95}{\parbox{.5\linewidth}{%
\begin{align*}
&\dec\GenB\bi{\hspace{0,5cm}X_1,\ldots,X_d}{Y_0;Y_1,\ldots,Y_d}\\
&=\sum_{\substack{k_1,\dots,k_d \geq 1 \\ m_0,\dots,m_d\geq 0}} \dec(b_0^{m_0}b_{k_1}b_0^{m_1}\dots b_{k_d}b_0^{m_d})Y_0^{m_0}X_1^{k_1-1}Y_1^{m_1}\dots X_d^{k_d-1}Y_d^{m_d} \\
&=\sum_{i=0}^d \sum_{\substack{k_1,\dots,k_d \geq 1 \\ m_0,\dots,m_d\geq 0}} \sum_{\substack{n_1+n_2=m_i \\ n_1,n_2\geq0}}(b_0^{m_0}b_{k_1}b_0^{m_1}\dots b_{k_i}b_0^{n_1}\otimes b_0^{n_2}b_{k_{i+1}}b_0^{m_{i+1}}\dots b_{k_d}b_0^{m_d})Y_0^{m_0}X_1^{k_1-1}Y_1^{m_1}\dots X_d^{k_d-1}Y_d^{m_d} 
\end{align*} }}
Reordering these sums, we obtain that this expression is equal to \\
\scalebox{0.95}{\parbox{.5\linewidth}{%
\begin{align*}
& \sum_{i=0}^d \sum_{\substack{k_1,\dots,k_d \geq 1 \\ m_0,\dots,m_{i-1},n_1,n_2,m_{i+1},\ldots,m_d\geq 0}} (b_0^{m_0}b_{k_1}b_0^{m_1}\dots b_{k_i}b_0^{n_1}\otimes b_0^{n_2}b_{k_{i+1}}b_0^{m_{i+1}}\dots b_{k_d}b_0^{m_d})Y_0^{m_0}X_1^{k_1-1}Y_1^{m_1}\\
&\hspace{9,5cm}\dots X_i^{k_i-1}Y_i^{n_1+n_2}X_{i+1}^{k_{i+1}-1}Y_{i+1}^{m_{i+1}}\dots X_d^{k_d-1}Y_d^{m_d}  \\
&=\sum_{i=0}^d \GenB\bi{\hspace{0,5cm} X_1,\ldots,X_i}{Y_0;Y_1,\ldots,Y_i}\otimes\GenB\bi{\hspace{0,5cm} X_{i+1},\ldots,X_d}{Y_i;Y_{i+1},\ldots,Y_d}
\end{align*} }} \\
\end{proof} \noindent
Next, we introduce a regularization, which assigns to each word in $\QB$ an element in $\QB^0$, such that the balanced quasi-shuffle product $\bst$ is preserved. The regularization process is similar to the ones for multiple zeta values given in \cite[Proposition 1]{ikz}. We will rephrase this regularization map in terms of generating series of words.
\begin{prop}  \label{reg map q-shuffle}
Let $T$ be a formal variable and extend the product $\bst$ by $\QQ[T]$-linearity to $\QB^0[T]$. The map
\begin{align*} 
\regT:\ \QB^0[T]&\to \QB,\\ wT^n&\mapsto w\bst b_0^{\bst n}
\end{align*}
is an algebra isomorphism for the balanced quasi-shuffle product $\bst$.
\end{prop} 
\begin{proof} For the surjectivity of $\regT$, we show that any word $w\in \QB$ is a polynomial in $b_0$ with coefficients in $\QB^0$.
Let $w=b_0^{m_0}b_{k_1}b_0^{m_1}\dots b_{k_d}b_0^{m_d}$ for some integers $k_1,\ldots,k_d\geq1$ and $m_0,\ldots,m_d\geq0$. We prove by induction on $m_0$, that $w=u+v\bst b_0$ for some $u\in \QB^0$ and $v\in \QB$, where all words in $v$ have weight smaller than $\wt(w)$. Then induction on the weight proves the claim.
The case $m_0=0$ is trivial, simply choose $u=w,\ v=0$. Next, calculate
\begin{align*}
b_0^{m_0-1}b_{k_1}b_0^{m_1}\dots b_{k_d}b_0^{m_d} \bst b_0= 
&\ m_0\ b_0^{m_0}b_{k_1}b_0^{m_1}\dots b_{k_d}b_0^{m_d}\\
&+\sum_{i=1}^{d} (m_i+1)\ b_0^{m_0-1}b_{k_1}b_0^{m_1}\dots b_{k_i}b_0^{m_i+1}b_{k_{i+1}}b_0^{m_{i+1}}\dots b_{k_d}b_0^{m_d} .
\end{align*}
Applying the induction hypotheses to every word in the second line leads to
\[w=\frac{1}{m_0}\left(u+\big(v+b_0^{m_0-1}b_{k_1}b_0^{m_1}\dots b_{k_d}b_0^{m_d}\big)\bst b_0\right)\]
for some $u\in \QB^0$ and $v\in \QB$, where $v+b_0^{m_0-1}b_{k_1}b_0^{m_1}\dots b_{k_d}b_0^{m_d}$ consists of words of weight smaller than $\wt(w)$. \\
Let $P\in \QB^0[T]\backslash\{0\}$ and write $P=wT^n+R$, where $w\in \QB^0\backslash\{0\}$ and $R\in \QB^0[T]$ is a polynomial of degree smaller than $n$. We have
\[\regT(P)=w\bst b_0^{\bst n}+\regT(R)=n!b_0^nw+\widetilde{w},\] where $\widetilde{w}\in \QB$ consists of words with at most $(n-1)$-times the letter $b_0$ at the beginning. We deduce $\regT(P)\neq 0$, thus $\regT$ is injective.
\end{proof} \noindent
In particular, by first applying the inverse $\regT^{-1}:\QB\to \QB^0[T]$ and then evaluating in $T=0$ yields the desired \emph{regularization map}
\begin{align} \label{reg map}
\reg:(\QB,\bst)\to(\QB^0,\bst).
\end{align}
By construction, it is an algebra morphism for the balanced quasi-shuffle product $\bst$.
\begin{thm} \label{reg on gen series} For all $d\geq1$, we have
\begin{align*}
\reg\GenB\bi{\hspace{0,5cm}X_1,\ldots,X_d}{Y_0;Y_1,\ldots,Y_d}=\GenBo\bi{X_1,\ldots,X_d}{Y_1-Y_0,\ldots,Y_d-Y_0}.
\end{align*}
\end{thm}
\begin{proof} We obtain from \cite[equation (5.2)]{ikz} applied to our setup that for each word $w=b_0^{m_0}b_{k_1}b_0^{m_1}\dots b_{k_d}b_0^{m_d}$ we have
\begin{align*}
\reg(w)&=(-1)^{m_0} b_{k_1}(b_0^{m_0}\bst b_0^{m_1}b_{k_2}b_0^{m_2}\dots b_{k_d}b_0^{m_d}) \\
&=(-1)^{m_0} \sum_{\substack{n_1+\dots+n_d=m_0 \\ n_1,\ldots,n_d\geq0}} \binom{m_1+n_1}{n_1}\dots \binom{m_d+n_d}{n_d}b_{k_1}b_0^{m_1+n_1}b_{k_2}b_0^{m_2+n_2}\dots b_{k_d}b_0^{m_d+n_d}.
\end{align*}
Hence we deduce
\begin{align} \label{7}
&\reg\GenB\bi{\hspace{0,5cm}X_1,\ldots,X_d}{Y_0;Y_1,\ldots,Y_d}\\
&=\sum_{\substack{k_1,\dots,k_d \geq 1 \\ m_0,\dots,m_d\geq 0}} (-1)^{m_0} \sum_{\substack{n_1+\dots+n_d=m_0 \\ n_1,\ldots,n_d\geq0}} \binom{m_1+n_1}{n_1}\dots \binom{m_d+n_d}{n_d}b_{k_1}b_0^{m_1+n_1}b_{k_2}b_0^{m_2+n_2}\dots b_{k_d}b_0^{m_d+n_d} \nonumber \\
&\hspace{11cm} \cdot Y_0^{m_0}X_1^{k_1-1}Y_1^{m_1}\dots X_d^{k_d-1}Y_d^{m_d}. \nonumber \\
\end{align} 
On the other hand, we compute
\begin{align} \label{8}
&\GenBo\bi{X_1,\ldots,X_d}{Y_1-Y_0,\ldots,Y_d-Y_0} \\
&= \sum_{\substack{k_1,\dots,k_d \geq 1 \\ m_1,\dots,m_d\geq 0}} \sum_{n_1=0}^{m_1}\dots \sum_{n_d=0}^{m_d} (-1)^{n_1+\dots+n_d}\binom{m_1}{n_1}\dots \binom{m_d}{n_d}b_{k_1}b_0^{m_1}\dots b_{k_d}b_0^{m_d}Y_0^{n_1+\dots +n_d}X_1^{k_1-1}Y_1^{m_1-n_1} \nonumber \\
&\hspace{13,3cm} \dots X_d^{k_d-1}Y_d^{m_d-n_d}. \nonumber
\end{align}
Comparing the coefficients of the monomial $Y_0^{m_0}X_1^{k_1-1}Y_1^{m_1}\dots X_d^{k_d-1}Y_d^{m_d}$ in \eqref{7} and \eqref{8} yields the claimed equality.
\end{proof} 
\begin{defi} \label{def dec^0} On the space $\QB^0$ define the \emph{regularized coproduct}
\[\dec^0=(\reg\otimes\reg)\circ\dec.\]
\end{defi} \noindent
Combining the formulas given in Proposition \ref{dec on GenB} and Theorem \ref{reg on gen series}, we also obtain a description of the regularized coproduct $\dec^0$ on the generating series $\GenBo$ of words in $\QB^0$.
\begin{prop} \label{dec^0 gen series B}
For all $d\geq1$, we have
\begin{align*}
\dec^0\GenBo\bi{X_1,\ldots,X_d}{Y_1,\ldots,Y_d}&=\sum_{i=0}^d \GenBo\bi{X_1,\ldots,X_i}{Y_1,\ldots,Y_i}\otimes\GenBo\bi{X_{i+1},\ldots,X_d}{Y_{i+1}-Y_i,\ldots,Y_d-Y_i}.
\end{align*}
\end{prop} 

\section{Comparison of the generating series of words in $\QYbi$ and $\QB^0$} \label{Comparison Ybi B}

\noindent
In this section, we will relate the product and coproduct formulas given for the generating series of words in $\QYbi$ and $\QB^0$ (Section \ref{gen series Ybi section}, \ref{gen series B section}, \ref{gen series B regularization}).
\begin{defi} For any bimould $M=(M_d)_{d\geq0}$ define the bimould $M^{\#_Y}=(M_d^{\#_Y})_{d\geq0}$ by
\[M_d^{\#_Y}\bi{X_1,\ldots,X_d}{Y_1,\ldots,Y_d}=M_d\bi{X_1,\ldots,X_d}{Y_1,Y_1+Y_2,\ldots,Y_1+\dots+Y_d},\qquad d\geq0.\]
\end{defi} \noindent
Evidently, the inverse operation is given by 
\[M_d^{\#_Y^{-1}}\bi{X_1,\ldots,X_d}{Y_1,\ldots,Y_d}=M_d\bi{X_1,\ldots,X_d}{Y_1,Y_2-Y_1,\ldots,Y_d-Y_{d-1}},\qquad d\geq0.\]
\begin{defi} \label{def varphi}
Let $\varphi_{\#}:\QYbi\to\QB^0$ be the $\QQ$-linear map implicitly defined by
\begin{align*}
(\varphi_{\#}\otimes\rho_{\Ybi})(\mathcal{W})_d\bi{X_1,\ldots,X_d}{Y_1,\ldots,Y_d}=\GenBo_d^{\#_Y}\bi{X_1,\ldots,X_d}{Y_1,\ldots,Y_d},
\end{align*}
i.e., the coefficient of $X_1^{k_1-1}\frac{Y_1^{m_1}}{m_1!}\dots X_d^{k_d-1}\frac{Y_d^{m_d}}{m_d!}$ in $\GenBo_d^{\#_Y}$ equals the image of the word $y_{k_1,m_1}\dots y_{k_d,m_d}$ under the map $\varphi_{\#}$.
\end{defi}
\begin{ex} We obtain for all $k_1,k_2\geq1,\ m_1,m_2\geq0$
\begin{align*}
\varphi_{\#}(y_{k_1,m_1})&=m_1! b_{k_1}b_0^{m_1}, \\
\varphi_{\#}(y_{k_1,m_1}y_{k_2,m_2})&= \sum_{n=m_2}^{m_1+m_2} \frac{m_1!n!}{(n-m_2)!}b_{k_1}b_0^{m_1+m_2-n}b_{k_2}b_0^n.
\end{align*}	
\end{ex}
\begin{thm} \label{varphi sharp alg hom}
The map $\varphi_{\#}:(\QYbi,\ast)\to (\QB^0,\bst)$ is an isomorphism of weight-graded algebras.
\end{thm}
\begin{proof} Evidently, $\varphi_{\#}$ is an isomorphism of graded vector $\QQ$-vector spaces. We will now show the compatibility of the products $\ast$ and $\bst$ on the level of generating series, this means we will prove that $\GenB^{\#_Y}$ satisfies the stuffle product formula given in Theorem \ref{stuffle Ybi gen series}. For all $d,d'\geq 1$ denote by $\shuffle(d,d')$ the set of all permutations $\sigma\in S_{d+d'}$ satisfying $\sigma(1)<\dots<\sigma(d)$, $\sigma(d+1)<\dots<\sigma(d+d')$. Moreover, write $\underline{Y_i}=Y_1+\dots+Y_i$ for $1\leq i\leq d$ and $\underline{Y_{d+j}}=Y_{d+1}+\dots+Y_{d+j}$ for $1\leq j \leq d'$. Then by Theorem \ref{balanced quasi-shuffle gen series}, we have modulo terms of lower depths
\begin{align} \label{1}
&\GenB^{\#_Y}\bi{X_1,\ldots,X_d}{Y_1,\ldots,Y_d} \GenB^{\#_Y}\bi{X_{d+1},\ldots,X_{d+d'}}{Y_{d+1},\ldots,Y_{d+d'}}  \\
&\hspace{4cm}\equiv\sum_{\sigma\in \shuffle(d,d')} \GenB\bi{X_{\sigma^{-1}(1)},\ldots,X_{\sigma^{-1}(d+d')}}{\underline{Y_{\sigma^{-1}(1)}}+\underline{Y_{\sigma_\mu^{-1}(1)}},\ldots,\underline{Y_{\sigma^{-1}(d+d')}}+\underline{Y_{\sigma_\mu^{-1}(d+d‘)}}}. \nonumber
\end{align}
Here we set 
\[\sigma_\mu^{-1}(k)=\begin{cases} \sigma^{-1}\big(\max\{n \mid \sigma^{-1}(n)>d,n<k\}\big), & 1\leq \sigma^{-1}(k)\leq d, \\ \sigma^{-1}\big(\max\{n \mid \sigma^{-1}(n)\leq d,\ n<k\}\big), &d+1\leq \sigma^{-1}(k)\leq d+d',\end{cases}\] 
and $\underline{Y_{\sigma_\mu^{-1}(k)}}=0$ if such a number $n$ does not exists. 
Observe that for every term in the right-hand side of \eqref{1} and for every $k=1,\ldots,d+d'$ the predecessor of the entry $\underline{Y_{\sigma^{-1}(k)}}+\underline{Y_{\sigma_\mu^{-1}(k)}}$ in the bi-index is given by $\underline{Y_{\sigma^{-1}(k)-1}}+\underline{Y_{\sigma_\mu^{-1}(k)}}$ (where $\underline{Y_{\sigma^{-1}(k)-1}}:=0$ if $k\in\{1,d+1\}$) and, we have $\underline{Y_n}-\underline{Y_{n-1}}=Y_n$. Thus, we deduce 
\begin{align} \label{2}
\GenB^{\#_Y^{-1}}\bi{X_{\sigma^{-1}(1)},\ldots,X_{\sigma^{-1}(d+d')}}{\underline{Y_{\sigma^{-1}(1)}}+\underline{Y_{\sigma_\mu^{-1}(1)}},\ldots,\underline{Y_{\sigma^{-1}(d+d')}}+\underline{Y_{\sigma_\mu^{-1}(d+d‘)}}}=\GenB\bi{X_{\sigma^{-1}(1)},\ldots,X_{\sigma^{-1}(d+d')}}{Y_{\sigma^{-1}(1)},\ldots,Y_{\sigma^{-1}(d+d')}}.
\end{align}
Combining the equations \eqref{1} and \eqref{2}, we get modulo terms of lower depths
\begin{align*} 
&\GenB^{\#_Y}\bi{X_1,\ldots,X_d}{Y_1,\ldots,Y_d} \GenB^{\#_Y}\bi{X_{d+1},\ldots,X_{d+d'}}{Y_{d+1},\ldots,Y_{d+d'}}\\
&\hspace{7cm}=\sum_{\sigma\in \shuffle(d,d')} \GenB^{\#_Y}\bi{X_{\sigma^{-1}(1)},\ldots,X_{\sigma^{-1}(d+d')}}{Y_{\sigma^{-1}(1)},\ldots,Y_{\sigma^{-1}(d+d')}} 
\end{align*}
Moreover, for every entry $\bi{X_j}{\underline{Y_j}+\underline{Y_{j'}}}$ in the lower depth terms coming from the third line in the recursive expression of $\bst$ given in Theorem \ref{balanced quasi-shuffle gen series}, the predecessor in the lower row is given by $\underline{Y_{j-1}}+\underline{Y_{j'-1}}$ (with $\underline{Y_{j-1}}=0$ for $j=1,d+1$). Therefore, the terms of lower depth equals the ones in the stuffle product formula on generating series of words given in Theorem \ref{stuffle Ybi gen series}. Altogether, this means the product 
\[\GenB^{\#_Y}\bi{X_1,\ldots,X_d}{Y_1,\ldots,Y_d} \GenB^{\#_Y}\bi{X_{d+1},\ldots,X_{d+d'}}{Y_{d+1},\ldots,Y_{d+d'}}\]
has an expression via the same recursive formula as for the stuffle product obtained in Theorem \ref{stuffle Ybi gen series}. Thus, $\varphi_{\#}$ is an algebra morphism.
\end{proof} 
\begin{cor} \label{symmetril and b-symmetril} A bimould $M=(M_d)_{d\geq0}$ is b-symmetril if and only if the bimould $M^{\#_Y}$ is symmetril.
\end{cor}
\begin{proof} This is a direct consequence of Theorem \ref{varphi sharp alg hom} and the definition of symmetrility and b-symmetrility (compare to Definition \ref{def symmetril}, \ref{def b-symmetril}).
\end{proof} \noindent
The algebra morphism $\varphi_{\#}:\QYbi\to\QB^0$ given in Theorem \ref{varphi sharp alg hom} is compatible with the coproducts $\dec$ and $\dec^0$.
\begin{prop} \label{compatibility dec coproducts} For all $d\geq1$, we have
\[(\varphi_{\#}\otimes\varphi_{\#})\circ \dec=\dec^0\circ \varphi_{\#}.\]
\end{prop}
\begin{proof} We will prove this on the level of generating series of words by using the formulas for $\dec$ and $\dec^0$ given in Proposition \ref{dec gen series Ybi} and \ref{dec^0 gen series B}. This allows us to straight-forwardly compute \\
\scalebox{0.95}{\parbox{.5\linewidth}{%
\begin{align*}
&(\#_Y\otimes\#_Y)\circ \dec \GenYbi\bi{X_1,\ldots,X_d}{Y_1,\ldots,Y_d}\\
&=\sum_{i=0}^d \GenYbi^{\#_Y}\bi{X_1,\ldots,X_i}{Y_1,\ldots,Y_i}\otimes\GenYbi^{\#_Y}\bi{X_{i+1},\ldots,X_d}{Y_{i+1},\ldots,Y_d} \\
&=\sum_{i=0}^d \GenYbi\bi{X_1,\ldots,X_i}{Y_1,Y_1+Y_2,\ldots,Y_1+\dots+Y_i}\otimes\GenYbi\bi{X_{i+1},\ldots,X_d}{Y_{i+1},Y_{i+1}+Y_{i+2},\ldots,Y_{i+1}+\dots+Y_d}.
\end{align*} }} \\
On the other hand, we have \\
\scalebox{0.95}{\parbox{.5\linewidth}{%
\begin{align*}
&\dec^0\GenBo^{\#_Y}\bi{X_1,\ldots,X_d}{Y_1,\ldots,Y_d}=\dec^0\GenBo\bi{X_1,\ldots,X_d}{Y_1,Y_1+Y_2,\ldots,Y_1+\dots+Y_d}\\
&=\sum_{i=0}^d \GenBo\bi{X_1,\ldots,X_i}{Y_1,Y_1+Y_2,\ldots,Y_1+\dots+Y_i}\otimes \GenBo\bi{X_{i+1},\ldots,X_d}{Y_{i+1},Y_{i+1}+Y_{i+2},\ldots,Y_{i+1}+\dots+Y_d}.
\end{align*} }} \\
By definition of the map $\varphi_{\#}$ (Definition \ref{def varphi}) we get the claimed formula.
\end{proof} \noindent
\begin{rem} Theorem \ref{varphi sharp alg hom} and Proposition \ref{compatibility dec coproducts} show that $\varphi_{\#}$ is an isomorphism compatible with the products and coproducts of $\QB^0$ and $\QYbi$. So since $(\QYbi,\ast,\dec)$ is a Hopf algebra, also $(\QB^0,\bst,\dec^0)$ must be a Hopf algebra. In particular, $\dec^0$ is coassociative and compatible with the product $\bst$.
\end{rem}
\noindent
There are two involutions defined on $\QB^0$ and $\QYbi$, which play an important role in the theory of multiple q-zeta values and which are compatible with the morphism $\varphi_{\#}$.
\begin{defi} \label{def tau} Let $\tau:\QB^0\to \QB^0$ be the involution given by $\tau(\1)=\1$ and
\[\tau(b_{k_1}b_0^{m_1}\dots b_{k_d}b_0^{m_d})=b_{m_d+1}b_0^{k_d-1}\dots b_{m_1+1}b_0^{k_1-1}\]
for all $k_1,\ldots,k_d\geq1,\ m_1,\ldots,m_d\geq0$.
\end{defi} \noindent
Extend the involution $\tau$ by $\QQ[X_1,Y_1,X_2,Y_2,\ldots]$-linearity to $\QB^0\llbracket X_1,Y_1,X_2,Y_2,\ldots\rrbracket $, then we obtain directly for each $d\geq0$ that
\[\tau\Bigg( \GenBo\bi{X_1,\ldots,X_d}{Y_1,\ldots,Y_d}\Bigg)=\GenBo\bi{Y_d,\ldots,Y_1}{X_d,\ldots,X_1}.\]
This motivates the following definition for bimoulds.
\begin{defi} \label{tau-invariant bimould}
For a bimould $M=(M_d)_{d\geq0}$, define the bimould $\tau(M)=(\tau(M)_d)_{d\geq0}$ by
\begin{align*} 
\operatorname{\tau}(M)_d\bi{X_1,\ldots,X_d}{Y_1,\ldots,Y_d}&=M_d\bi{Y_d,\ldots,Y_1}{X_d,\ldots,X_1}.
\end{align*}
We call a bimould $M$ \emph{$\tau$-invariant} if $\tau(M)=M$.
\end{defi} \noindent
The second involution defined on the algebra $\QYbi$ is originally given in terms of bimoulds, since its expression is much easier in this case.
\begin{defi} \label{def swap} For a bimould $M=(M_d)_{d\geq0}$, define the bimould $\swap(M)=(\swap(M)_d)_{d\geq0}$ as
\[\swap(M)_d\bi{X_1,\ldots,X_d}{Y_1,\ldots,Y_d}=M_d\bi{Y_1+\dots+Y_d,Y_1+\dots+Y_{d-1},\ldots,Y_1}{X_d,X_{d-1}-X_d,\ldots,X_1-X_2}.\]
A bimould $M$ is called \emph{swap invariant} if $\swap(M)=M$.
\vspace{0,2cm} \\
Define the involution $\swap:\QYbi\to \QYbi$ implicitly by 
\[(\swap\otimes\rho_{\Ybi})(\mathcal{W})_d\bi{X_1,\ldots,X_d}{Y_1,\ldots,Y_d}=\swap(\GenYbi)_d\bi{X_1,\ldots,X_d}{Y_1,\ldots,Y_d},\]
i.e., the coefficient of $X_1^{k_1-1}\frac{Y_1^{m_1}}{m_1!}\dots X_d^{k_d-1}\frac{Y_d^{m_d}}{m_d!}$ in $\swap(\GenYbi)_d$ is the image of the word $y_{k_1,m_1}\dots y_{k_d,m_d}$ under the map $\swap$.
\end{defi} \noindent
For example, we obtain
\begin{align} \label{swap on coefficients} 
\swap(y_{k_1,m_1})&=\frac{m_1!}{(k_1-1)!}y_{m_1+1,k_1-1}, \\
\swap(y_{k_1,m_1}y_{k_2,m_2})&=\sum_{u=0}^{m_1}\sum_{v=0}^{k_2-1} \frac{(-1)^v}{u!v!} \frac{m_1!}{(k_1-1)!}\frac{(m_2+u)!}{(k_2-1-v)!}y_{m_2+1+u,k_2-1-v}y_{m_1+1-u,k_1-1+v}. \nonumber
\end{align} 
In higher depths, it is hard to give an explicit formula for the involution $\swap$ on the algebra $\QYbi$, see for example \cite[Remark 3.14]{bi}.
\begin{thm} \label{Hopf iso Ybi B}
The map
\[\varphi_{\#}:(\QYbi,\ast,\dec)\to(\QB^0,\bst,\dec^0)\]
is an isomorphism of weight-graded Hopf algebras satisfying $\varphi_{\#}\circ \swap=\tau\circ\varphi_{\#}$.
\end{thm}
\begin{proof} The map $\varphi_{\#}$ is an isomorphism of weight-graded Hopf algebras by Theorem \ref{varphi sharp alg hom} and Proposition \ref{compatibility dec coproducts}. Moreover, one verifies directly on the level of generating series of words
\begin{align*}
\#_Y\circ\swap \GenYbi\bi{X_1,\ldots,X_d}{Y_1,\ldots,Y_d}&=\GenYbi^{\#_Y}\bi{Y_1+\dots+Y_d,Y_1+\dots+Y_{d-1},\ldots,Y_1}{X_d,X_{d-1}-X_d,\ldots,X_1-X_2}\\
&=\GenYbi\bi{Y_1+\dots+Y_d,Y_1+\dots+Y_{d-1},\ldots,Y_1}{X_d,X_{d-1},\ldots,X_1},
\end{align*}
and
\begin{align*}
\tau \GenBo^{\#_Y}\bi{X_1,\ldots,X_d}{Y_1,\ldots,Y_d}&=\tau\GenBo\bi{X_1,\ldots,X_d}{Y_1,Y_1+Y_2,\ldots,Y_1+\dots+Y_d}\\
&=\GenBo\bi{Y_1+\dots+Y_d,Y_1+\dots+Y_{d-1},\ldots,Y_1}{X_d,X_{d-1},\ldots,X_1}.
\end{align*}
By definition of the map $\varphi_{\#}$ (Definition \ref{def varphi}), we deduce that $\varphi_{\#}\circ \swap=\tau\circ\varphi_{\#}$.
\end{proof} \noindent
An immediate consequence of Theorem \ref{Hopf iso Ybi B} is the following.
\begin{cor} \label{swap invariant and tau invariant} A bimould $M=(M_d)_{d\geq0}$ is $\tau$-invariant if and only if the bimould $M^{\#_Y}$ is swap invariant.
\end{cor}

\section{The algebra of multiple  $\operatorname{q}$-zeta values} \label{qMZV}

\noindent
We introduce the algebra of multiple q-zeta values $\Zq$ and explain its relations to multiple zeta values and polynomial functions on partitions. As an application of the previous general results, we will present a particular nice spanning set for $\Zq$, the balanced multiple q-zeta values.
\begin{defi} \label{def generic qMZV} (\cite{bk})
To integers $s_1\geq 1,\ s_2,\ldots,s_l\geq 0$ and polynomials $R_1\in t\QQ[t],\\ R_2,\ldots,R_l\in$ $\QQ[t]$, associate the \emph{generic multiple q-zeta value}
\[\zeta_q(s_1,...,s_l;R_1,...,R_l)=\sum_{n_1>\dots>n_l>0} \frac{R_1(q^{n_1})}{(1-q^{n_1})^{s_1}}\cdots \frac{R_l(q^{n_l})}{(1-q^{n_l})^{s_l}}\in \QQ\llbracket q\rrbracket.\]
\end{defi} \noindent
The assumptions $s_1\geq 1$ and $R_1\in t\QQ[t]$ are necessary for convergence. 
\begin{prop} \label{limit generic qMZV} (\cite[p. 6]{bk}) For integers $s_1\geq 2,\ s_2,\ldots,s_l\geq 1$ and polynomials $R_1\in t\QQ[t],\ R_2,\ldots,R_l\in$ $\QQ[t]$, one has
\[\lim_{q\to 1} (1-q)^{s_1+\dots+s_l}\zeta_q(s_1,...,s_l;R_1,...,R_l)=R_1(1)\cdots R_l(1)\zeta(s_1,...,s_l).\]
\vspace{-1cm} \\ \qed
\end{prop} \noindent
Following H. Bachmann and U. Kühn (\cite{bk}), we consider the following kind of multiple q-zeta values.
\begin{defi} \label{def Zq} 
Define the $\QQ$-vector space
\[\Zq=\spanQ\!\big\{\zeta_q(s_1,...,s_l;R_1,...,R_l)\ \big|\ l\geq 0,s_1\geq 1,\ s_2,...,s_l\geq 0,\deg(R_j)\leq s_j\big\},\]
where we set $\zeta_q(\emptyset;\emptyset)=1$.
\end{defi} \noindent 
The space $\Zq$ contains all models for multiple q-zeta values studied in the literature, so this can be seen as a model-free approach to them. For $\zeta_q(s_1;R_1),\zeta_q(s_2;R_2)\in \Zq$, the usual power series multiplication reads
\[\zeta_q(s_1;R_1)\cdot \zeta_q(s_2;R_2)=\zeta_q(s_1,s_2;R_1,R_2)+\zeta_q(s_2,s_1;R_2,R_1)+\zeta_q(s_1+s_2;R_1R_2).\]
Since $\deg(R_1R_2)\leq s_1+s_2$, the product is also an element in $\Zq$. Similar computations for arbitrary multi indices show that $\Zq$ is an associative, commutative algebra. Moreover, the algebra $\Zq$ contains the algebra $\quasimod$ of quasi-modular forms with rational coefficients (as introduced in \cite{kz}) and is closed under the well-known derivation $q\frac{\diff}{\diff q}$.
\begin{rem}
The additional assumption on the degree of the polynomials $R_j$ in Definition \ref{def Zq} can be justified by its relations to polynomial functions on partitions. More precisely, the original definition of the space $\Zq$ was given by H. Bachmann as the span of the bi-brackets (\cite{ba}), which can be seen as generating series of monomial functions on partitions. In \cite[cf (1.6)]{bi} it is shown that the space $\Zq$ is exactly the image of the polynomial functions on partitions under the q-bracket. \\
Let $\lambda=(1^{m_1}2^{m_2}3^{m_3}\dots)$ be a partition of some natural number $N$ of length $d$, i.e., the multiplicities $m_i\in \ZZ_{\geq0}$ are nonzero only for finitely many indices $i_1,\ldots,i_d$ and one has $\sum_{i\geq 1} m_ii=N$. A polynomial $f\in \QQ[X_1,\ldots,X_d,Y_1,\ldots,Y_d]$ can be evaluated at the partition $\lambda$ by
\[f(\lambda)=f(i_1,\ldots,i_d,m_{i_1},\ldots,m_{i_d}).\]
Denote by $\mathcal{P}(N,d)$ the set of all partitions of $N$ of length $d$. Then we associate to a polynomial $f\in \QQ[X_1,\ldots,X_d,Y_1,\ldots,Y_d]$ the generating series 
\[\operatorname{Gen}_f(q)=\sum_{N\geq1}\left(\sum_{\lambda\in \mathcal{P}(N,d)} f(\lambda)\right)q^N.\]
Whenever $f\in\QQ[X_1,\ldots,X_d,Y_1,\ldots,Y_d]$ is chosen to be a monomial, then $\operatorname{Gen}_f(q)$ is equal to a bi-bracket of depth $d$, more details are given in \cite[Theorem 1.3]{bri}. Therefore, as a consequence of \cite[Theorem 2.3]{bk} we obtain that the space $\Zq$ is spanned by the generating series $\operatorname{Gen}_f(q)$ with $f\in\QQ[X_1,\ldots,X_d,Y_1,\ldots,Y_d]$, $d\geq0$.
This indicates that the elements in $\Zq$ should be invariant under some involution corresponding to conjugation of partitions.
\end{rem}

\section{Combinatorial bi-multiple Eisenstein series} \label{CMES}

\noindent
We present the combinatorial bi-multiple Eisenstein series constructed in \cite{bb}. They form a spanning set for the space $\Zq$, so they are a particular model for multiple q-zeta values. Combinatorial bi-multiple Eisenstein series satisfy a weight-graded product formula and are invariant under a weight-homogeneous involution. Therefore, they should induce a grading on $\Zq$, which extends the weight-grading of the algebra $\quasimod$ of quasi-modular forms with rational coefficients. Their construction is inspired by the Fourier expansion of multiple Eisenstein series (\cite[Theorem 1.4]{ba2}). Thus, the combinatorial bi-multiple Eisenstein series give a natural connection between the space $\Zq$ and multiple Eisenstein series. In particular, they should give a description of all relations between multiple Eisenstein series. We will recall the construction of the combinatorial bi-multiple Eisenstein series as given in \cite{bb}.
\begin{defi} \label{def b}
By the work of G. Racinet (\cite{ra}) and also the combination of the works of V. G. Drinfeld (\cite{dr}) and H. Furusho (\cite{fu}) there exists a rational solution to the extended double shuffle equations
\[\beta(k_1,\ldots,k_d)\in \QQ,\qquad k_1,\ldots,k_d\in\ZZ_{\geq1},\ k_1>1,\]
such that
\begin{align} \label{beta condition in depth 1}
\beta(k)=-\frac{B_k}{2k!},\ k \text{ even}, \qquad \beta(k)=0,\ k\text{ odd}.
\end{align}
Denote by $\beta_\ast(k_1,\ldots,k_d),\ k_1,\ldots,k_d\geq1$, the corresponding stuffle regularized elements and define for all $d\geq1$
\begin{align} \label{def mould b} \mathfrak{b}_d(X_1,\ldots,X_d)=\sum_{k_1,\ldots,k_d\geq1} \beta_\ast(k_1,\ldots,k_d)X_1^{k_1-1}\dots X_d^{k_d-1}.
\end{align}
Moreover set $\mathfrak{b}_0=1$, then $\mathfrak{b}=(\mathfrak{b}_d)_{d\geq0}$ is a mould with coefficients in $\QQ$.
\end{defi} \noindent
In general, a solution to the extended double shuffle equations is not unique. In the following, we will fix the mould $\mathfrak{b}$ and the whole construction of the combinatorial bi-multiple Eisenstein series will depend on this choice. 
\begin{defi}
Define the bimould $\mathfrak{b}=(\mathfrak{b}_d)_{d\geq0}$ by $\mathfrak{b}_0=1$ and for $d\geq1$ by
\begin{align*}
\mathfrak{b}_d\bi{X_1,\dots,X_d}{Y_1,\dots,Y_d} = \sum_{0\leq i\leq  j \leq d} \gamma_i  \mathfrak{b}_{j-i}(Y_1+\dots+Y_{j-i},\dots,Y_1+Y_2,Y_{1}) \mathfrak{b}_{d-j}(X_{j+1},\dots,X_d),
\end{align*} 
where the coefficients $\gamma_i$ are defined by
\begin{align*}
\sum_{i\geq0} \gamma_i T^i = \exp\left( \sum_{n\geq2} \frac{(-1)^{n+1}}{n} \frac{B_n}{2 n!} T^n \right).
\end{align*}
\end{defi} \noindent
Independent of the shape of the mould $\mathfrak{b}$, this construction will always yield a swap invariant bimould. Since the coefficients of the mould $\mathfrak{b}$ satisfy the extended double shuffle relations, we obtain that the bimould $\mathfrak{b}$ is also symmetril.
\begin{defi}
Define the bimould $\widetilde{\mathfrak{b}}=(\widetilde{\mathfrak{b}}_d)_{d\geq 0}$ by $\widetilde{\mathfrak{b}}_0=1$ and
\begin{align*}
\widetilde{\mathfrak{b}}_d\bi{X_1,\ldots,X_d}{Y_1,\ldots, Y_d}= \sum_{i=0}^d \frac{(-1)^i}{2^i i!} \mathfrak{b}_{d-i}\bi{X_{i+1},\ldots,X_d}{-Y_1,\ldots, -Y_{d-i}},\qquad d\geq1.
\end{align*}
For each $u\geq 1$, let $\mathfrak{L}^{(u)}=(\mathfrak{L}^{(u)}_d)_{d\geq0}$ be the bimould given by $\mathfrak{L}^{(u)}_0=1$ and 
\begin{align*}
&\mathfrak{L}_d^{(u)}\bi{X_1,\ldots, X_d}{Y_1,\ldots, Y_d}=\sum_{j=1}^d \mathfrak{b}_{j-1}\bi{X_1-X_j,\ldots,X_{j-1}-X_j}{Y_1,\ldots,Y_{j-1}}L_u\bi{X_j}{Y_1+\dots+ Y_d} \\
&\hspace{10cm}\cdot \widetilde{\mathfrak{b}}_{d-j}\bi{X_d-X_j,\ldots, X_{j+1}-X_j}{Y_d,\ldots,Y_{j+1}},
\end{align*}
where the power series $L_u\binom{X}{Y}$ is defined by
\[L_u\bi{X}{Y}=\frac{\exp(X+uY)q^u}{1-\exp(X)q^u},\qquad u\geq1.\]
\end{defi} \noindent
It can be shown (\cite[Lemma 6.20]{bb}) that the bimould $\mathfrak{L}^{(u)}$ is symmetril, the proof uses the additional conditions on the depth $1$ terms of $\mathfrak{b}$ given in \eqref{beta condition in depth 1}.
\begin{defi} Define the bimould $\mathfrak{g}^\ast=(\mathfrak{g}^\ast_d)_{d\geq0}$ by  $\mathfrak{g}^\ast_0=1$ and 
\begin{align*}
\mathfrak{g}^\ast_d\bi{X_1,\ldots,X_d}{Y_1,\ldots, Y_d} = \sum_{\substack{1 \leq j \leq d\\0 = d_0< d_1 < \dots < d_{j-1} < d_j = d\\ u_1 > \dots > u_j > 0}}  \prod_{i=1}^j \mathfrak{L}^{(u_i)}_{d_i-d_{i-1}} \bi{X_{d_{i-1}+1},\ldots,X_{d_i}}{Y_{d_{i-1}+1},\ldots,Y_{d_i}}.
\end{align*}
\end{defi} \noindent
Since the bimould $\mathfrak{L}^{(u)}$ is symmetril, also $\mathfrak{g}^\ast$ is symmetril. This implication is independent of the explicit shape of $\mathfrak{L}^{(u)}$.
\begin{defi} \label{def G}
Let $\mathfrak{G}=(\mathfrak{G}_d)_{d\geq 0}$ be the mould product of $\mathfrak{g}^\ast$ and $\mathfrak{b}$, i.e., one has $\mathfrak{G}_0=1$ and for $d\geq1$
\begin{align*}
\Gcmes{d}{X_1,\ldots,X_d}{Y_1,\ldots,Y_d}=\sum_{i=0}^d \mathfrak{g}_i^\ast\bi{X_1,\ldots,X_i}{Y_1,\ldots,Y_i}\mathfrak{b}_{d-i}\bi{X_{i+1},\ldots,X_d}{Y_{i+1},\ldots,Y_d}.
\end{align*}
\end{defi} \noindent
The main result of \cite{bb} was the following.
\begin{thm} \label{G symmetril, swap invariant} \cite[Theorem 6.5.]{bb} The bimould $\mathfrak{G}$ is symmetril and swap invariant. \qed 
\end{thm} \noindent
As a product of symmetril bimoulds, $\mathfrak{G}$ is symmetril. To obtain the swap invariance, one decomposes $\mathfrak{G}$ into a sum of bimoulds, for which the swap invariance can be shown.
\begin{defi} \label{def cmes} For $k_1,\dots,k_d \geq 1$, $m_1,\dots,m_d\geq 0$ define the \emph{combinatorial bi-multiple Eisenstein series} $G\binom{k_1,\ldots,k_d}{m_1,\ldots,m_d}$ to be the (normalized) coefficients of the bimould $\mathfrak{G}$, 
\begin{align*}
\Gcmes{d}{X_1,\ldots,X_d}{Y_1,\ldots,Y_d}=\sum_{\substack{k_1,\dots,k_d \geq 1 \\ m_1,\dots,m_d\geq 0}}\cmes{k_1,\ldots,k_d}{m_1,\ldots,m_d} X_1^{k_1-1}\frac{Y_1^{m_1}}{m_1!}\dots X_d^{k_d-1}\frac{Y_d^{m_d}}{m_d!}.
\end{align*}
We refer to the number $k_1+\dots+k_d+m_1+\dots+m_d$ as the weight of $G\binom{k_1,\ldots,k_d}{m_1,\ldots,m_d}$ and to the number $d$ as its depth. Moreover, the elements 
\[G(k_1,\dots,k_d) = \cmes{k_1,\ldots,k_d}{0,\ldots,0}, \qquad k_1,\ldots,k_d\geq1,\] 
are called the \emph{combinatorial multiple Eisenstein series}. 
\end{defi} 
\begin{ex} \label{CMES in depth 1,2}
In depth $1$, one obtains
\[\Gcmes{1}{X_1}{Y_1}=\mathfrak{b}_1\bi{X_1}{Y_1}+\sum_{u>0}L_u\binom{X_1}{Y_1}\]
and thus for $k\geq 1,\ m\geq 0$
\[\cmes{k}{m}=-\delta_{m,0}\frac{B_k}{2k!}-\delta_{k,1}\frac{B_{m+1}}{2(m+1)}+
\frac{1}{(k-1)!}\sum_{u,v>0}u^mv^{k-1}q^{uv}.\]
In particular, the combinatorial Eisenstein series $G(k)$, $k\geq 2$ even, are exactly the classical Eisenstein series of weight $k$ with rational coefficients (expressed in their Fourier expansion). The combinatorial bi-Eisenstein series $G\binom{k}{m}$, $k+m\geq 2$ even, is essentially the $m$-th derivative of the classical Eisenstein series $G(k)$ and hence is also contained in the algebra $\quasimod$ of quasi-modular forms with rational coefficients.
\end{ex} 
\begin{prop} \label{cmes spanning set} (\cite[Proposition 6.15.]{bb}) The combinatorial bi-multiple Eisenstein series form a spanning set of $\Zq$. \qed 
\end{prop} \noindent
Theorem \ref{G symmetril, swap invariant}, the definition of the symmetrility and swap (compare to Definition \ref{def symmetril}, \ref{def swap}), and Proposition \ref{cmes spanning set} imply the following.
\begin{cor} \label{quasish cmes} There is a surjective, swap invariant algebra morphism
\begin{align*}
G:(\QYbi,\ast)&\to\Zq, \\
y_{k_1,m_1}\dots y_{k_d,m_d}&\mapsto\cmes{k_1,\ldots,k_d}{m_1,\ldots,m_d}.
\end{align*}
\end{cor} 
\begin{con} \label{all relations in Zq CMES} (\cite[Remark 6.11.]{bb})
All algebraic relations in $\Zq$ are a consequence of the stuffle product formula and the swap invariance of the combinatorial bi-multiple Eisenstein series.
\end{con} \noindent
If Conjecture \ref{all relations in Zq CMES} holds, then the algebra $\Zq$ is graded by weight
\begin{align} \label{grading CMES}
\Zq=\bigoplus_{w\geq0} \Zq^{(w)}.
\end{align}
Here $\Zq^{(w)}$ denotes the subspace of $\Zq$ spanned by all combinatorial bi-multiple Eisenstein series of weight $w$. This follows immediately from the observation that the stuffle product and also the swap operator are homogeneous in weight.

\section{Balanced multiple $\operatorname{q}$-zeta values} \label{balanced qMZV}

\noindent
We will apply the results from Section \ref{Comparison Ybi B} to obtain a new spanning set of the algebra $\Zq$. They will satisfy very explicit and simple relations, which are homogeneous in weight. Thus, this new spanning set should induce a weight-grading on $\Zq$. 
\begin{defi} \label{def balanced qMZV} Define the bimould $\mathfrak{B}=(\mathfrak{B}_d)_{d\geq0}$ with coefficients in $\Zq$ by $\mathfrak{B}_0=1$ and for $d\geq 1$ by 
\[\mathfrak{B}_d\bi{X_1,\ldots,X_d}{Y_1,\ldots,Y_d}=\mathfrak{G}^{\#_Y^{-1}}_d\bi{X_1,\ldots,X_d}{Y_1,\ldots,Y_d},\] 
where $\mathfrak{G}$ is the bimould of generating series of the combinatorial bi-multiple Eisenstein series (Definition \ref{def G}). The \emph{balanced multiple q-zeta values} $\zq(s_1,\ldots,s_l)$, $s_1\geq1$, $s_2,\ldots,s_l\geq0$, are the coefficients of the bimoulds $\mathfrak{B}$,
\begin{align*}
\Gzq{d}{X_1,\ldots,X_d}{Y_1,\ldots,Y_d}=\sum_{\substack{k_1,\dots,k_d \geq 1 \\ m_1,\dots,m_d\geq 0}} \zq(k_1,\{0\}^{m_1},\ldots,k_d,\{0\}^{m_d})X_1^{k_1-1}Y_1^{m_1}\dots X_d^{k_d-1}Y_d^{m_d}, \quad d\geq1.
\end{align*}
Equivalently, the balanced multiple q-zeta values can be defined as
\[\zq(s_1,\ldots,s_l)=G(\varphi_{\#}^{-1}(b_{s_1}\dots b_{s_l})),\]
where $G:\QYbi\to \Zq$ denotes the algebra morphism induced by the combinatorial bi-multiple Eisenstein series given in Corollary \ref{quasish cmes}.
\end{defi} 
\begin{rem} The definition of the balanced multiple q-zeta values has many similarities with Zudilin's definition of the multiple q-zeta brackets in terms of the bi-brackets (\cite[eq (8)]{zu}). 
\end{rem} \noindent
Since the bimould $\mathfrak{G}$ of generating series of the combinatorial bi-multiple Eisenstein series is symmetril and swap invariant (Theorem \ref{G symmetril, swap invariant}), an immediate consequence of Corollary \ref{symmetril and b-symmetril} and \ref{swap invariant and tau invariant} is the following.
\begin{prop} The bimould $\mathfrak{B}$ is b-symmetril and $\tau$-invariant.
\end{prop} \noindent
These symmetries of the bimould $\mathfrak{B}$ yield the following properties of balanced multiple q-zeta values.
\begin{thm} \label{balanced qMZV quasi-shuffle morphism} There is a $\tau$-invariant, surjective algebra morphism
\begin{align*}
\zq:(\QB^0,\bst)&\to\Zq, \\
b_{s_1}\dots b_{s_l}&\mapsto\zq(s_1,\ldots,s_l).
\end{align*}
\end{thm} 
\begin{proof}
This is an immediate consequence of the b-symmetrility and $\tau$-invariance of the bimould $\mathfrak{B}$ (see Definition \ref{def b-symmetril}, \ref{tau-invariant bimould}) together with the observation that $\#_Y$ is a bijection.
\end{proof} \noindent 
As a reformulation of Conjecture \ref{all relations in Zq CMES}, we expect the following.
\begin{con} \label{all relations in Zq balanced qMZV} All algebraic relations in $\Zq$ are a consequence of the balanced quasi-shuffle product formula and the $\tau$-invariance of the balanced multiple q-zeta values.
\end{con} \noindent
Observe that the balanced quasi-shuffle product, as well as the $\tau$-invariance, are both completely explicit on the level of words of multi indices and easy to compute. If Conjecture \ref{all relations in Zq balanced qMZV} holds, then the algebra $\Zq$ is graded by weight
\[\Zq=\bigoplus_{w\geq0} \Zq^{(w)},\]
where $\Zq^{(w)}$ is the subspace of $\Zq$ spanned by all balanced multiple q-zeta values of weight $w$. This conjectural grading coincides with the one of the combinatorial bi-multiple Eisenstein series given in \eqref{grading CMES}. \\
It was conjectured in \cite[Conjecture 4.3]{ba} in a slightly different setup that the space $\Zq$ is spanned by the elements $\zq(k_1,\ldots,k_d)$, $k_1,\ldots,k_d\geq1$. Some partial results towards this conjecture are obtained in \cite[Section 6]{bu}, \cite[Proposition 4.4, 5.9]{ba} and \cite[Theorem 5.3]{vl} and a general proof is announced in \cite{bbi}.
\begin{ex} 
1. In depth $1$ the balanced multiple q-zeta values coincide with the combinatorial bi-multiple Eisenstein series up to multiplication with certain factorials. Thus, we deduce from Example \ref{CMES in depth 1,2} that for all $k\geq1,\ m\geq0$
\[\zq(k,\{0\}^m)=-\delta_{m,0}\frac{B_k}{2k!}-\delta_{k,1}\frac{B_{m+1}}{2(m+1)!}+
\frac{1}{(k-1)!m!}\sum_{u,v>0}u^mv^{k-1}q^{uv}.\]
So, the element $\zq(k,\{0\}^m)$ is essentially equal to the $m$-th derivative of the Eisenstein series $G_k$. If $k_1,\ldots,k_d\geq1$, then the balanced multiple q-zeta value $\zq(k_1,\ldots,k_d)$ equals the combinatorial multiple Eisenstein series $G(k_1,\ldots,k_d)$. In particular, the balanced multiple q-zeta values give a very natural extension and explicit description of conjectural all relations between multiple Eisenstein series.
\vspace{0,2cm} \\
2. Recall that the mould $\mathfrak{b}$ in depth $2$ is given by $\mathfrak{b}(X_1,X_2)=\sum_{k_1,k_2\geq1} \beta_\ast(k_1,k_2) X_1^{k_1-1}X_2^{k_2-1}$ (compare to \eqref{def mould b}). Direct calculations show that
\begin{align*}
\zq(2,3)&=\beta_\ast(2,3)-\frac{1}{48}\sum_{u,v>0} v^2q^{uv}+\frac{1}{2}\sum_{\substack{u_1>u_2>0 \\ v_1,v_2>0}} v_1v_2^2q^{u_1v_1+u_2v_2}, \\
\zeta(2,0,3)&=\frac{1}{2}\sum_{\substack{u_1>u_2>0 \\ v_1,v_2>0}} u_1v_1v_2^2q^{u_1v_1+u_2v_2}-\frac{1}{2}\sum_{\substack{u_1>u_2>0 \\ v_1,v_2>0}} u_2v_1v_2^2q^{u_1v_1+u_2v_2}.
\end{align*}
An explicit construction for the numbers $\beta_\ast(k_1,k_2)$ is given in \cite[Section 6]{gkz}, in this case one obtains $\beta_\ast(2,3)=0$. Further constructions for these rational numbers are given in \cite{bro}, \cite{ec-mzv}.
\vspace{0,2cm} \\
3. The quasi-modular forms are contained in the algebra $\Zq$. In particular, the modular discriminant $\Delta(q)=q\prod_{n\geq1}(1-q^n)^{24}$ is a $\QQ$-linear combination of balanced multiple q-zeta values. For example, using the exotic relation given in \cite[Example 2.48 (ii)]{co} we obtain
\begin{align*}
\frac{1}{43200}\Delta(q)= &\ 240\zq(4,4,4)-63\zq(9,3)+183\zq(8,4)-\frac{675}{2}\zq(7,5)+\frac{89}{2}\zq(6,6)-378\zq(5,7)\\
&+183\zq(4,8).
\end{align*}
\end{ex} 

\section{Further properties of the balanced multiple $\operatorname{q}$-zeta values} \label{further properties balanced qMZV}

\noindent
Similar to the case of the Schlesinger-Zudilin multiple q-zeta values studied by K. Ebrahimi-Fard, D. Manchon, and J. Singer (\cite{si}, \cite{ems}), the balanced multiple q-zeta values possess a description in terms of an alphabet with two letters. We will compute the limits $q\to1$ of a certain kind of balanced multiple q-zeta values and obtain that those are elements in the algebra of multiple zeta values. Finally, we will describe the derivation $q\frac{\diff}{\diff q}$ on the algebra $\Zq$ in terms of the balanced multiple q-zeta values.
\begin{defi} \label{qsh on Q(p,y)} Let $\bsh$ be the product on the non-commutative free algebra $\Qnca{p,y}$ recursively defined by $\1\bsh w=w \bsh \1=w$ and 
\begin{align*}
(yu)\bsh v&=u\bsh (yv)=y(u\bsh v), \\
(pu)\bsh(pv)&=p(u\bsh pv)+p(pu\bsh v)+\begin{cases} p(u\bsh v), &\text{ if } u=y\tilde{u} \text{ and } v=y\tilde{v}, \\ 0 &\text{ else}\end{cases} \nonumber
\end{align*}
for all $u,v,w\in \Qnca{p,y}$.  
\end{defi} \noindent
The involution $\tau$ (Definition \ref{def tau}) can be also defined on the algebra $\Qnca{p,y}$.
\begin{defi}  \label{def tau on Q(p,y)} Let $\tau$ be the anti-automorphism on $\Qnca{p,y}$ given by $\tau(\1)=\1$, $\tau(p)=y$, and $\tau(y)=p$, i.e., one has for all $k_1,\ldots,k_d\geq 1,\ m_0,\ldots,m_d\geq 0$
\begin{align*}
\tau(y^{m_0}p^{k_1}y^{m_1}\dots p^{k_d}y^{m_d})=p^{m_d}y^{k_d}\dots p^{m_1}y^{k_1}p^{m_0}.\end{align*}
\end{defi} \noindent
The involution $\tau$ connects the balanced quasi-shuffle product $\bst$ and the product $\bsh$. To make this precise, consider the canonical embedding
\begin{align*} 
i:\QB&\hookrightarrow \Qnca{p,y}, \\
b_{s_1}\dots b_{s_l}&\mapsto p^{s_1}y\dots p^{s_l}y. \nonumber
\end{align*} 
Observe that we have $\tau(i(w))=i(\tau(w))$ for all $w\in \QB^0$.
\begin{prop} \label{tau relates graded qsh and qst} For all $u,v\in \QB$, we have
\[i(u\bst v)=\tau(\tau\circ i(u)\bsh \tau\circ i(v)).\]
\end{prop}
\begin{proof} Let $u=b_{s_1}\dots b_{s_l}$ and $v=b_{r_1}\dots b_{r_k}$ be words in $\QB$ and $s_l,r_k\geq 1$. Using the recursive definition of $\bst$ from the right, we obtain 
\begin{align*}
i(u\bst v)&=i(b_{s_1}\dots b_{s_{l-1}}\bst b_{r_1}\dots b_{r_k})p^{s_l}y+i(b_{s_1}\dots b_{s_l}\bst b_{r_1}\dots b_{r_{k-1}})p^{r_k}y\\
& \hspace{0,35cm}+i(b_{s_2}\dots b_{s_l}\bst b_{r_2}\dots b_{r_k})p^{s_l+r_k}y.
\end{align*} 
On the other hand, applying the definition of $\tau$ and $i$ gives \\
\scalebox{0.99}{\parbox{.5\linewidth}{%
\begin{align*}
&\tau(\tau\circ i(u)\bsh \tau\circ i(v))=\tau(py^{s_l}\dots py^{s_1} \bsh py^{r_k}\dots py^{r_1}) \\
&=\tau\Big(py^{s_l}(py^{s_{l-1}}\dots py^{s_1}\bsh py^{r_k}\dots py^{r_1})+py^{r_k}(py^{s_l}\dots py^{s_1}\bsh py^{r_{k-1}}\dots py^{r_1})\\
&\hspace{0,35cm}+py^{s_l+r_k}(py^{s_{l-1}}\dots py^{s_1}\bsh py^{r_{k-1}}\dots py^{r_1})\Big) \\
&=\tau\Big(\tau\circ i(b_{s_1}\dots b_{s_{l-1}})\bsh \tau \circ i(b_{r_1}\dots b_{r_k})\Big)p^{s_l}y+\tau\Big(\tau\circ i(b_{s_1}\dots b_{s_l})\bsh \tau\circ i(b_{r_1}\dots b_{r_{k-1}})\Big) p^{r_k}y \\
&\hspace{0,35cm} +\tau\Big(\tau\circ i(b_{s_1}\dots b_{s_{l-1}})\bsh \tau \circ i(b_{r_1}\dots b_{r_{k-1}})\Big)p^{s_l+r_k}y.
\end{align*} }} \\
So induction on the depth implies the claim. Next, assume that $s_l=0$ and $r_k\geq 0$. Then we obtain 
\begin{align*}
i(u\bst v)&=i(b_{s_1}\dots b_{s_{l-1}}\bst b_{r_1}\dots b_{r_k})y+i(b_{s_1}\dots b_{s_l}\bst b_{r_1}\dots b_{r_{k-1}})p^{r_k}y,
\end{align*}
and on the other hand, using again the definition of $\tau$ and $i$
\begin{align*}
&\tau(\tau\circ i(u)\bsh \tau\circ i(v))=\tau(py^{s_l}\dots py^{s_1} \bsh py^{r_k}\dots py^{r_1}) \\
&=\tau\Big(p(py^{s_{l-1}}\dots py^{s_1}\bsh py^{r_k}\dots py^{r_1})+py^{r_k}(py^{s_l}\dots py^{s_1}\bsh py^{r_{k-1}}\dots py^{r_1})\Big) \\
&=\tau\Big(\tau\circ i(b_{s_1}\dots b_{s_{l-1}})\bsh \tau \circ i(b_{r_1}\dots b_{r_k})\Big)y+\tau\Big(\tau\circ i(b_{s_1}\dots b_{s_l})\bsh \tau\circ i(b_{r_1}\dots b_{r_{k-1}})\Big) p^{r_k}y.
\end{align*}
Again, induction on the depth implies the claim.
\end{proof} \noindent
By Proposition \ref{tau relates graded qsh and qst}, there is injective algebra morphism
\begin{align*} \label{morphism tau circ i}
\tau\circ i:(\QB,\bst)&\hookrightarrow(\Qnca{p,y},\bsh), \\
b_{s_1}\dots b_{s_l}&\mapsto py^{s_l}\dots py^{s_1}. \nonumber
\end{align*}
In particular, the restriction of $\bsh$ to $\operatorname{im}(\tau\circ i)=\QQ\1+p\Qnca{p,y}$ can be seen as a quasi-shuffle product. Denote
\[\Qnca{p,y}^0=\QQ\1+p\Qnca{p,y}y.\]
\begin{thm} \label{grSZ q-shuffle} There is a $\tau$-invariant, surjective algebra morphism 
\begin{align*}
(\Qnca{p,y}^0,\bsh) &\to \Zq, \\
p^{s_1}y\dots p^{s_l}y&\mapsto\zq(s_1,\ldots,s_l).
\end{align*}
\end{thm} \noindent
\begin{proof} 	
Observe that by definition 
\begin{align} \label{10} \zq(i(u))=\zq(u) \quad \text{ for all }u\in \QB^0.
\end{align}  
So as the balanced multiple q-zeta values form a spanning set of $\Zq$, we obtain surjectivity. Since $\tau(i(u))=i(\tau(u))$ for each $u\in \QB^0$, we deduce the $\tau$-invariance from Theorem \ref{balanced qMZV quasi-shuffle morphism}. Finally, we prove that the map is an algebra morphism for $\bsh$. For $s_1,r_1\geq 1,\ s_2,\ldots, s_l,r_2,\ldots,r_k\geq 0$, we obtain by applying the $\tau$-invariance and \eqref{10} \\
\scalebox{0.99}{\parbox{.5\linewidth}{%
\begin{align*}
\zq(py^{s_l}\dots py^{s_1})\zq(py^{r_k}\dots py^{r_1})&=\zq(p^{s_1}y\dots p^{s_l}y)\zq(p^{r_1}y\dots p^{r_k}y)= \zq(b_{s_1}\dots b_{s_l})\zq(b_{r_1}\dots b_{r_k}). 
\end{align*} }} \\
Since $\zq$ is an algebra morphism for the balanced quasi-shuffle product $\bst$, we deduce 
\begin{align*}
\zq(py^{s_l}\dots py^{s_1})\zq(py^{r_k}\dots py^{r_1})&=\zq(b_{s_1}\dots b_{s_l}\bst b_{r_1}\dots b_{r_k}).
\end{align*} 
Applying Proposition \ref{tau relates graded qsh and qst} and then again the $\tau$-invariance gives
\begin{align*} 
\zq(py^{s_l}\dots py^{s_1})\zq(py^{r_k}\dots py^{r_1})
&=\zq\big(\tau\big(\tau\circ i(b_{s_1}\dots b_{s_l})\bsh \tau\circ i(b_{r_1}\dots b_{r_k})\big)\big) \\
&=\zq\big(\tau\circ i(b_{s_1}\dots b_{s_l})\bsh \tau\circ i(b_{r_1}\dots b_{r_k})\big).
\end{align*} 
Finally, applying the definition of $\tau$ and $i$ yields the desired formula
\begin{align*} 
\zq(py^{s_l}\dots py^{s_1})\zq(py^{r_k}\dots py^{r_1})
&=\zq\big(py^{s_l}\dots py^{s_1}\bsh py^{r_k}\dots py^{r_1}\big).
\end{align*} 
\end{proof} \noindent
Next, we will compute the limit of the balanced multiple q-zeta values for $q\to1$ for some special multi indices. 
\begin{prop} \label{non-regularized limit balanced qMZV} For a word $w=b_{\varepsilon_1}\dots b_{\varepsilon_n}b_{k_1}\dots b_{k_d}$ in $\QB^0$, where $\varepsilon_1,\ldots,\varepsilon_n\in\{0,1\}$, $\varepsilon_n=0$ and $k_1,\ldots,k_d\in \mathbb{Z}_{\geq1}$, $k_1\geq2$, we have
\[\lim_{q\to 1} (1-q)^{\wt(w)}\zq(w)=\zeta(x_{\varepsilon_n}\dots x_{\varepsilon_1})\zeta(y_{k_1}\dots y_{k_d}).\]
\end{prop} \noindent
Here we set as explained in \eqref{stuffle MZV}, \eqref{shuffle MZV},
\begin{align*}
\zeta(x_0^{k_1-1}x_1\dots x_0^{k_d-1}x_1)=\zeta(k_1,\ldots,k_d), \qquad \qquad \zeta(y_{k_1}\dots y_{k_d})=\zeta(k_1,\ldots,k_d),
\end{align*}
for all $k_1,\ldots,k_d\in\mathbb{Z}_{\geq1}$, $k_1\geq2$.
\begin{proof}
Let $\mathfrak{z}=(\mathfrak{z}_d)_{d\geq0}$ be the mould of the multiple zeta values, so for each $d\geq1$ we have
\[\mathfrak{z}(X_1,\ldots,X_d)=\sum_{k_1\geq2,k_2,\ldots,k_d\geq1} \zeta(k_1,\ldots,k_d)X_1^{k_1-1}\dots X_d^{k_d-1}.\]
Similar to \cite[Definition 1.3]{bi} define the conjugated multiple zeta values $\xi(m_1,\ldots,m_d)$, $m_1,\ldots,m_{d-1}\geq0,\ m_d\geq1$, by
\begin{align} \label{def conjuagted mzv}
\mathfrak{z}(Y_1+\dots+Y_d,\ldots,Y_1+Y_2,Y_1)=\sum_{m_1,\ldots,m_{d-1}\geq0,m_d\geq1} \xi(m_1,\ldots,m_d) \frac{Y_1^{m_1}}{m_1!}\dots \frac{Y_d^{m_d}}{m_d!}.
\end{align}
By \cite[Theorem 4.18]{bi} and \cite[Remark 6.18]{bb}, the combinatorial bi-multiple Eisenstein series satisfy
\begin{align*}
\lim_{q\to 1} (1-q)^{m_1+\dots+m_j+j+k_{j+1}+\dots+k_d} \cmes{\hspace{0,3cm}1,\ldots,1,\hspace{0,2cm}k_{j+1},\ldots,k_d}{\hspace{-0,5cm}m_1,\ldots,m_j,\hspace{0,2cm}0,\ldots,0}=\xi(m_1,\ldots,m_j)\zeta(k_{j+1},\ldots,k_d)
\end{align*}
for all $1\leq j \leq d$. We deduce from the definition of the balanced multiple q-zeta values (Definition \ref{def balanced qMZV}) that we have for $\varepsilon_1,\ldots,\varepsilon_n\in\{0,1\}$, $\varepsilon_n=0$ and $k_1,\ldots,k_d\in \mathbb{Z}_{\geq1}$, $k_1\geq2$
\begin{align} \label{5}
\lim_{q\to 1} (1-q)^{n+k_1+\dots+k_d} \zq(b_{\varepsilon_1}\dots b_{\varepsilon_n}b_{k_1}\dots b_{k_d})&=\lim_{q\to 1}(1-q)^{n+k_1+\dots+k_d} G(\varphi_{\#}^{-1}(b_{\varepsilon_1}\dots b_{\varepsilon_n}b_{k_1}\dots b_{k_d})) \nonumber \\
&=\xi(\varphi_{\#}^{-1}(b_{\varepsilon_1}\dots b_{\varepsilon_n}))\zeta(k_1,\ldots,k_d).
\end{align}
Here $\xi(\varphi_{\#}^{-1}(b_{\varepsilon_1}\dots b_{\varepsilon_n}))$ means that we first have to apply $\varphi_{\#}$ to $b_{\varepsilon_1}\dots b_{\varepsilon_n}$, which gives an element in $\Qnca{y_{1,m}\mid m\geq0}$, and then we have to identify $\xi(y_{1,m_1}\dots y_{1,m_d})=\xi(m_1,\ldots,m_d)$. On the level of generating series, this means we have to apply $\#_Y^{-1}$ (Definition \ref{def varphi}), therefore we substitute $Y_1\mapsto Y_1,\ Y_2\mapsto Y_2-Y_1,\ldots, Y_d\mapsto Y_d-Y_{d-1}$ in the left-hand side of \eqref{def conjuagted mzv} and obtain
\begin{align} \label{6}
\mathfrak{z}(Y_d,Y_{d-1},\ldots,Y_1)
&=\sum_{m_1,\ldots,m_{d-1}\geq1,m_d\geq2} \zeta(m_d,\ldots,m_1)Y_1^{m_1-1}\dots Y_d^{m_d-1}. 
\end{align}
Combining \eqref{5} and \eqref{6} gives the desired formula.
\end{proof}
\begin{rem} \label{regularized limit balanced qMZV} Similar to \cite[Definition 4.17]{bi}, one could define a regularized limit 
for the balanced multiple q-zeta values by a certain regularization process in the algebra $\QB$ and obtains for each word $w\in \QB$
\[\lim_{q\to1}\textbf{}\!^* (1-q)^{\wt(w)}\zq(w)=\sum_{\substack{uv=w \\ u=b_{\varepsilon_1}\dots b_{\varepsilon_n},\ \varepsilon_i\in \{0,1\} \\ v=b_{k_1}\dots b_{k_d},\ k_i\in \mathbb{N}}} \zeta^{\shuffle}(x_{\varepsilon_n}\dots x_{\varepsilon_1})\zeta^{\ast}(y_{k_1}\dots y_{k_d}).\]
Here $\zeta^{\shuffle}$ and $\zeta^\ast$ denote the shuffle and stuffle regularized multiple zeta value map as given in \eqref{stuffle MZV}, \eqref{shuffle MZV}. In particular, the regularized limit of the balanced multiple q-zeta value $\zq(w)$ vanishes whenever the word $w$ cannot be decomposed as $w=uv$ with $u\in \Qnca{b_0,b_1}$ and $v\in \Qnca{b_i\mid i\geq1}$.
\end{rem}

\noindent
Finally, we want to describe the derivation $q\frac{\diff}{\diff q}$ on $\Zq$ in terms of the balanced multiple q-zeta values. The obtained formula can be seen as a weight-graded version of the derivation formula for the Schlesinger-Zudilin multiple q-zeta values given in \cite[Theorem 4.1]{si}.
\begin{prop} We have for all $s_1\geq1,s_2,\ldots,s_l\geq0$ that
\begin{align*}
q\frac{\diff}{\diff q}\zq(s_1,\ldots,s_l)=\sum_{i=1}^l\sum_{j=i}^l s_i\zq(s_1,\ldots,s_i+1,\ldots,s_j,0,s_{j+1},\ldots,s_l).
\end{align*}
\end{prop} \noindent
Observe that the derivation $q\frac{\diff}{\diff q}$ is a homogeneous operator increasing the weight by $2$.
\begin{proof} By \cite[Proposition 6.29]{bb} the generating series of the combinatorial bi-multiple Eisenstein series satisfy 
\begin{align*}
q\frac{\diff}{\diff q} \Gcmes{d}{X_1,\ldots,X_d}{Y_1,\ldots,Y_d}=\sum_{i=1}^d \frac{\partial}{\partial X_i}\frac{\partial}{\partial Y_i} \Gcmes{d}{X_1,\ldots,X_d}{Y_1,\ldots,Y_d}, \qquad d\geq1.
\end{align*}
By definition of the balanced multiple q-zeta values (Definition \ref{def balanced qMZV}), we obtain that \\
\scalebox{0.975}{\parbox{.5\linewidth}{%
\begin{align} \label{3}
\sum_{\substack{k_1,\dots,k_d \geq 1 \\ m_1,\dots,m_d\geq 0}} q\frac{\diff}{\diff q} \zq(k_1,\{0\}^{m_1},\ldots,k_d,\{0\}^{m_d})X_1^{k_1-1}Y_1^{m_1}X_2^{k_2-1}(Y_1+Y_2)^{m_2}\dots X_d^{k_d-1}(Y_1+\dots+Y_d)^{m_d} 
\end{align} }} \\
must be equal to \\
\scalebox{0.975}{\parbox{.5\linewidth}{%
\begin{align} \label{4}
&\sum_{\substack{k_1,\dots,k_d \geq 1 \\ m_1,\dots,m_d\geq 0}} \sum_{i=1}^d \zq(k_1,\{0\}^{m_1},\ldots,k_d,\{0\}^{m_d}) \frac{\partial}{\partial X_i}\frac{\partial}{\partial Y_i} X_1^{k_1-1}Y_1^{m_1}X_2^{k_2-1}(Y_1+Y_2)^{m_2} \\
& \hspace{12cm} \dots X_d^{k_d-1}(Y_1+\dots+Y_d)^{m_d} \nonumber \\
&=\sum_{\substack{k_1,\dots,k_d \geq 1 \\ m_1,\dots,m_d\geq 0}} \sum_{i=1}^d\sum_{j=i}^d (k_i-1)m_j\zq(k_1,\{0\}^{m_1},\ldots,k_d,\{0\}^{m_d})X_1^{k_1-1}Y_1^{m_1}X_2^{k_2-1}(Y_1+Y_2)^{m_2} \nonumber \\
&\hspace{2,3cm}\dots X_i^{k_i-2}(Y_1+\dots+Y_i)^{m_i}\dots X_j^{k_j-1}(Y_1+\dots+Y_j)^{m_j-1}\dots X_d^{k_d-1}(Y_1+\dots+Y_d)^{m_d}. \nonumber
\end{align} }} \\
Coefficient comparison at $X_1^{k_1-1}Y_1^{m_1}X_2^{k_2-1}(Y_1+Y_2)^{m_2}\dots X_d^{k_d-1}(Y_1+\dots+Y_d)^{m_d} $ in \eqref{3} and \eqref{4} yields 
\begin{align*}
&q\frac{\diff}{\diff q}\zq(k_1,\{0\}^{m_1},\ldots,k_d,\{0\}^{m_d})\\
&\hspace{2cm}=\sum_{i=1}^d\sum_{j=i}^d k_i(m_j+1)\zq(k_1,\{0\}^{m_1},\ldots,k_i+1,\{0\}^{m_i},\ldots,k_j,\{0\}^{m_j+1},\ldots,k_d,\{0\}^{m_d}),
\end{align*}
which is equivalent to the claimed formula.
\end{proof}

\vspace{1cm} 
\paragraph{\textbf{Outlook.}} In forthcoming articles we will formalize the balanced multiple q-zeta values and obtain an algebra generated by symbols satisfying the balanced quasi-shuffle product and $\tau$-invariance. Conjecturally, this algebra is equipped with a Hopf algebra structure where the coproduct is given by a generalization of Goncharov's coproduct. A first step towards this is a Lie algebra structure obtained in \cite{bu}. Moreover, this algebra of symbols should give a complete description of all relations in the algebra $\Zq$ and by Theorem \ref{Hopf iso Ybi B} also a description of all relations between multiple Eisenstein series. By applying Racinet's ideas for formal multiple zeta values (\cite{ra}) to this new algebra, we will obtain a formal version of the limit computation given in Proposition \ref{non-regularized limit balanced qMZV}. By \cite{bi} this map could be also seen as a formal version of taking the constant term of q-series.

\end{document}